\documentclass[psamsfonts]{amsart}
\usepackage{amssymb,amsfonts,lacromay}
\usepackage[all,arc]{xy}
\usepackage{enumerate}
\usepackage{mathrsfs}
\usepackage{stmaryrd}
\usepackage{stackrel,tikz}

\theoremstyle{thm}
\newtheorem{thm}{Theorem}[section]

\newtheorem{prop}[thm]{Proposition}
\newtheorem{lem}[thm]{Lemma}

\theoremstyle{definition}
\newtheorem{defn}[thm]{Definition}

\newtheorem{ouch0}[thm]{Counterexample}
\newtheorem{exmp}[thm]{Example}

\newtheorem{asses}[thm]{Assumptions}
\newtheorem{rem}[thm]{Remark}

\newtheorem{warn}[thm]{Warning}

\makeatletter
\let\c@equation\c@thm
\makeatother
\numberwithin{equation}{section}

\makeatletter
\ifx\SK@label\undefined\let\SK@label\label\fi
 \let\your@thm\@thm
 \def\@thm#1#2#3{\gdef\currthmtype{#3}\your@thm{#1}{#2}{#3}}
 \def\mylabel#1{{\let\your@currentlabel\@currentlabel\def\@currentlabel
  {\currthmtype~\your@currentlabel}
 \SK@label{#1@}}\label{#1}}
 \def\myref#1{\ref{#1@}}
\makeatother

\newcommand{\Co}{{\bf{Co}\sV}}

\renewcommand{\t}[1]{\text{#1}}
\newcommand{\til}[1]{\widetilde{#1}}

\title{Enriched model categories in equivariant contexts}

\author{Bertrand Guillou}
\email{bertguillou@uky.edu}
\address{Department of Mathematics,
University of Kentucky, 
Lexington, KY 40506}

\author{J.P. May}
\email{may@math.uchicago.edu}
\address{Department of Mathematics,
The University of Chicago,
Chicago, IL 60637}

\author{Jonathan Rubin}
\email{jrubin@math.uchicago.edu}
\address{Department of Mathematics,
The University of Chicago,
Chicago, IL 60637}

\date{}

\begin{document}

\begin{abstract}  We give a general framework of equivariant model  category theory.
Our groups $G$, called Hopf groups, are suitably defined group objects in any well-behaved symmetric monoidal 
category $\sV$.  For any $\sV$, a discrete group $G$ gives a
Hopf group, denoted  $I [G]$.  When $\sV$ is cartesian monoidal, the Hopf groups
are just the group objects in $\sV$.  When $\sV$ is the category of modules over a commutative
ring $R$, $I[G]$ is the group ring $R[G]$ and the general Hopf groups are the cocommutative 
Hopf algebras over $R$.  We show how all of the usual constructs of equivariant homotopy
theory, both categorical and model theoretic,  generalize to Hopf groups for any $\sV$.
This opens up some quite elementary unexplored mathematical territory, while systematizing
more familiar terrain.  
\end{abstract}

\maketitle

\tableofcontents

The model theoretic role of presheaf categories in enriched contexts is studied in general in \cite{GM}.  
This paper explores how that general theory plays out in equivariant contexts. We begin with 
an exposition of uniform categorical foundations and then turn to the homotopy theory. 
The context is more general and perhaps less familiar than expected.  It opens up some quite 
natural unexplored territory.

As a preamble, we start in \S\ref{GTop} with a well understood motivating example, the equivalence of model 
categories of $G$-spaces 
with presheaves of spaces defined on
the orbit category of $G$.  We view that as a template for generalization. In \S\ref{Hopf}, we ignore model categories and 
develop a coherent categorical context.  Briefly, following \cite{RW} and other precursors, we identify Hopf groups $G$ as the appropriate 
groups in a well-behaved symmetric monoidal category $\sV = (\sV, \otimes, I )$.  For a well-behaved
category $\sM$ enriched in $\sV$ and a Hopf group $G$ in $\sV$, we show how to develop a theory of 
$G$-objects in $\sM$, including among other desiderata the fixed point and orbit objects associated 
to $G$-objects. These lie  in $\sM$, and then the most obvious presheaves in sight take values in $\sM$ 
rather than in the enriching  category $\sV$, which is where it is usually most useful to have them.   We shall 
see that if the $G$-objects in $\sM$ have a model as  presheaves with values in $\sM$ and the objects of $\sM$ are modeled as presheaves in $\sV$,
then the $G$-objects in $\sM$ also have a model as 
presheaves with values in $\sV$.

In any cartesian monoidal category $\sV$, every object is a 
cocommutative comonoid in a unique way via the diagonal $\DE$. Therefore we may define group objects in $\sV$ exactly as we do in $\mathbf{Set}$.

Now let $\sV = (\sV,\otimes, I )$ be a cosmos, or a good enriching category. As we recall in \S\ref{Hopf},
the category $\Co$ of cocommutative comonoids in $\sV$ is cartesian monoidal. If $C$ and $D$ are in $\Co$, then so is  $C\otimes D$, and it is the categorical product of  $C$ and $D$ in $\Co$; with its evident {comultiplication} and counit, the unit object $I$ of $\sV$ is the unit object in $\Co$. We view groups in the symmetric monoidal category $(\Co,\otimes, I )$  to be the most natural generalization of groups in cartesian monoidal categories,
and we introduce the term Hopf group as the generic name for $\Co$-groups. We shall be more explicit in
\S\ref{Hgrp}.
If $\sV$ is cartesian monoidal, then $\sV = \Co$ and Hopf groups are {group objects in $\sV$, which we call} 
$\sV$-groups.  In general, we replace
$\sV$ by the cartesian monoidal category $\Co$ to define Hopf groups. When $\sV = R$-$\bf{Mod}$ for a commutative ring $R$, these are the cocommutative Hopf algebras over $R$.  We can think of  monoids in $\Co$ as cocommutative $\sV$-bialgebras and Hopf groups in $\sV$ as 
cocommutative $\sV$-Hopf algebras. 

Up to language, this much should be reasonably standard.   What is not standard is to see how thoroughly the usual constructions for $G$-objects in a category, 
where $G$ is a discrete group, generalize to give analogous constructions for $G$-objects in a category $\sM$
enriched in $\sV$, where $G$ is a Hopf group in $\sV$.   
This suggests a generalization of equivariant homotopy theory to the context of $G$-objects in $\sM$ for any Hopf group $G$ and $\sV$-category $\sM$.  The idea is that perspectives natural in equivariant homotopy theory apply just as well to actions of Hopf groups in general.  We have in mind new directions in homological algebra and 
stable homotopy theory, 
starting with the category {$\sV = R\t{-}\mathbf{Mod}$} of $R$-chain complexes 
and a  cocommutative Hopf algebra $A$ over $R$ in the former case.  For a very particular example, with $R = \bF_p$ we can take $A$ to be the mod 
$p$ Steenrod algebra.

This gives  a broad generalization of classical group actions. 
In general, for a Hopf group $G$, let $G\sV$ be the category 
of $G$-objects in $\sV$ 
and $G$-maps between them.  In the original version of this paper, by the  first two authors, we treated several equivariant 
contexts as different.  The present more categorical vantage point, promulgated by the third author, gives a common generalization.  

One detail from the original paper argues persuasively for the change of perspective.  
For any cosmos $\sV$, there is an evident strong symmetric monoidal functor 
$I[-]$ from sets to $\sV$, given by $I[S] = \amalg_S I$.
Applied to a discrete group $G$, it gives a Hopf group $I[G]$. 
Starting from a discrete group $G$ 
and its $\sV$-group ring $I[G]$, the homotopically correct version of the fixed point $\sV$-orbit category of $I[G]$  is the full 
$\sV$-subcategory of $G\sV$  with objects the  $I[G]/H$, where $H$ runs through the closed subgroups of $I[G]$ as defined
in \S\ref{closed} below; see \myref{GObjAsPrshvs}.
We could instead restrict to the full $\sV$-subcategory of $G\sV$ with objects of
the form $I[G/H]$, where $H$ is an ordinary subgroup of $G$.  A third choice is to restrict to the non-full $\sV$-subcategory with objects 
the $I[G/H]$ and morphism objects $I[G{\bf{Set}}(G/K, G/H)]$, which is what we might think of first.  Here the set $G{\bf{Set}}(G/K, G/H)$ of $G$-maps $G/K \rtarr G/H$ can be 
identified with $(G/H)^K$.  The difference between the latter two choices is seen in the inclusion
$$  I[(G/H)^K] \subset I[G/H]^K. $$
This is the identity in the usual cartesian monoidal examples $\sV$, but in general it is not.  For example, it is not
when $\sV = R$-$\bf{Mod}$, in which case $I = R$ and $R[G]$ is the usual group ring.  In this example,
the (closed) subgroups of $R[G]$ are the sub Hopf algebras over $R$, not just those of the form 
$R[H]$ for a subgroup $H$ of $G$.  The present perspective works for Hopf groups in any cosmos $\sV$ with no more difficulty 
than the original perspective on $\sV$-groups in cartesian monoidal categories $\sV$. 

While this idea is elementary and natural, it suggests new and as yet unexplored perspectives.  In particular, in equivariant homotopy theory, the most natural weak equivalences are certainly not the mere underlying weak equivalences, which are altogether too naive.\footnote{These are sometimes called Borel equivalences, since Borel homology, in contrast to Bredon homology,  is invariant under such naive equivalences.}  However, in homological algebra for DG-modules over Hopf algebras or, more generally, 
over DG-Hopf algebras, the weak equivalences used to date are the quasi-isomorphisms, which are the mere underlying weak equivalences.  

With the categorical perspective of \S\ref{Hopf} on hand, we turn to equivariant model categories in \S\ref{equiv2}. 
There we fix a set $\sF$ of (closed) subgroups of a Hopf group $G$ in $\sV$ and let $\sO_{\sF}$ be the full $\sV$-subcategory of $G\sV$ whose objects are the orbits $G/H$ in $G\sV$.   We assume that $\sV$ and $\sM$ have suitably related model structures.  We then have the evident analogues in $G\sV$ and $G\sM$ of the $\sF$-equivalences in {$G$-spaces}.  We show in Theorems \ref{GM}, \ref{GPM}, and \ref{GObjAsPrshvs} that under quite general conditions the given model structure on $\sM$ induces Quillen equivalent $\sF$-model structures on $G\sM$ and on the functor (or presheaf) category $\mathbf{Fun} (\sO_{\sF}^{op},\sM)$, in precise analogy with the comparison of model categories of $G$-spaces dealt with in \S\ref{GTop}.  

To model $G\sM$ by a category of presheaves with values in $\sV$ rather than $\sM$, 
we assume in \S\ref{GVModel} that $\sM$ itself is Quillen equivalent to a presheaf category\footnote{As in \cite[1.5]{GM},
we shall use $\mathbf{Fun}$ for enriched functor categories (functors being covariant) and reserve $\mathbf{Pre}$ for enriched categories 
of presheaves with values in $\sV$ (presheaves being contravariant).} 
 $\mathbf{Pre}(\sD,\sV)$ 
and that the results of \S\ref{GMModel} apply to $\sM$.  We show in \myref{neat} that $G\sM$ is then also Quillen equivalent to a presheaf model category  
$\mathbf{Pre} (\sO_{\sF}\otimes\sD, \sV)$.  In contrast to earlier results, the domain category $\sO_{\sF}\otimes\sD$ for the presheaves in $\sV$ needed to model $G\sM$ is generally {\em not} a full subcategory of $G\sV$.

We show how the general theory specializes to DG Hopf algebras in \S\ref{Hopfalg}.   
Letting $A$ denote a cocommutative DG hopf algebra over a commutative ring $R$,
we show that the results of \S\ref{equiv2} apply to give a Quillen equivalence between $A\text{-}\mathbf{Mod}$ with its $\sF$-model structure and the presheaf category $\mathbf{Pre} (\sO_{\sF}, R\text{-}\mathbf{Mod})$, where $\sO_{\sF}$ is the category of orbits, denoted $A/\!/B$, for $B$ in some chosen set $\sF$ of sub Hopf-algebras of $A$.  The theory here seems especially intriguing, and we hope to see it followed up in later work.

We take $\sV$ to be the category of simplicial sets in \S\ref{secsset}.  Here Hopf groups are simplicial groups $G$, but we can say little in that generality.  Our theory necessarily must focus on commutation relations between orbits and fixed points. This only works well simplicially when $G$ is a discrete group viewed as a constant simplicial set.  For such a $G$, taking $\sV = \sM = s\mathbf{Set}$, the results of \S\ref{equiv2} apply to give the well-known Quillen equivalence between $G s\mathbf{Set}$ with its $\sF$-model structure and the presheaf category $\mathbf{Pre} (\sO_{\sF}, s\mathbf{Set}) $, where $\sO_{\sF}$ is the category of orbits $G/H$ 
for $H\in \sF$.
As we explain in \S\ref{Mouch}, to generalize $\sM$ to a suitable category enriched in $s\mathbf{Set}$ is more problematic.  Restricting to a finite group $G$, we obtain a general theorem under appropriate 
assumptions on $\sM$ that ensure the required commutation of orbits and fixed points.

The reader will likely notice that other examples in algebraic geometry and category theory will work in much the same way as in the two quite different contexts just described.

We give several categorical elaborations in \S\ref{Catcat}.  We give a general definition of a ``closed'' subgroup of a Hopf group in \S\ref{closed}, we explain when $\sF$ deserves to be called a ``family'' of subgroups in \S\ref{families}, we elaborate on the double enrichment present in equivariant contexts in \S\ref{double}, and we give alternative perspectives on the basic adjunction $(\bT,\bU)$ relating functor categories to categories of $G$-objects in \S\ref{TU}.

We thank Emily Riehl for catching errors and for many helpful comments, and we thank a helpful referee 
for wading through a less satisfactory earlier draft.
This  work  was  partially  supported  by Simons Collaboration Grant No. 282316 held by the first author.

\section{Preamble: enriched model categories of $G$-spaces}\label{GTop}

We describe the motivating example \cite{DK, EHCT, Pia} for a general 
theory relating equivariant categories to presheaf categories.
Since the example is specified topologically, we use topological enrichments.
We work in $\sU$, the category of compactly generated weak Hausdorff spaces.

Let $G$ be a topological group. We understand subgroups of $G$ to be closed.
Let $\sF$ be a nonempty family of subgroups of $G$, so that $\{e\}\in \sF$ and subconjugates of
groups in $\sF$ are in $\sF$. 

\begin{rem} While it is standard to restrict attention to families, for the theory here
$\sF$ could be any (nonempty) set of subgroups of $G$.
\end{rem}

Specializing the general theory in \cite{GM}, 
take $\sV$ there to be the cartesian monoidal category $\sU$ 
with its standard Quillen model structure. The generating cofibrations $\mathcal{I}$ are
the cells $S^{n-1}\rtarr D^{n}$ and the generating acyclic cofibrations 
$\mathcal{J}$ are the inclusions $i_0\colon D^n\rtarr D^n\times I$, where $n\geq 0$.  

To be pedantically precise, we let $\ul{G\sU}(X,Y)$ denote the space of $G$-maps 
$X\rtarr Y$ and let $G\sU(X,Y)$ denote its underlying set.   The spaces $\ul{G\sU}(X,Y)$
are the hom objects that give $G\sU$ its enrichment in $\sU$.  

We may view $G\sU$ as a closed cartesian monoidal category, with $G$ acting diagonally 
on cartesian products. Its objects are the $G$-spaces. For $G$-spaces $X$ and $Y$, the 
internal hom $G$-space $\ul{\sU}_G(X,Y)$ is the $G$-space of all maps $X\rtarr Y$, with 
$G$ acting by conjugation. We emphasize that although $\ul{\sU}_G(X,Y)$ is defined using all maps $X\rtarr Y$, with our
understanding that it is a functor taking values in $G$-spaces it is
only functorial with respect to $G$-maps $X\rtarr X'$ and $Y\rtarr Y'$.   These $G$-spaces give the internal hom 
for the closed structure on $G\sU$. We let
$\sU_G(X,Y)$ denote the underlying set of $\ul{\sU}_G(X,Y)$, but  we emphasize that 
this notion of underlying set is of no great mathematical interest and will not generalize to 
our new context.  

Observe that for any $G$-space $X$, the space $X^G$ can be identified 
with the space $\ul{G\sU}(\ast,X)$.  We have identifications
{\begin{equation}\label{Gfixed}
 \ul{G\sU}(X,Y) =  \ul{\sU}_G(X,Y)^G
\end{equation}
and therefore 
\begin{equation}\label{enrichset}  
{G\sU}(X,Y) = {\sU}(\ast, \ul{\sU}_G(X,Y)^G). 
\end{equation}
That much will generalize to arbitrary $\sV$.  Here we also have 
\begin{equation}\label{enrichsettoo}
G\sU(X,Y) = \sU_G(X,Y)^G,
\end{equation}}
but that will only generalize under hypotheses on $\sV$ that are not usually satisfied.

It is usual to regard $G\sU(X,Y)$ as the space rather than just the set of $G$-maps $X\rtarr Y$,
and to write $\sU_G(X,Y)$ for the $G$-space rather than just the $G$-set of maps
$X\rtarr Y$. When we enrich a category in spaces or $G$-spaces, as here, it seems reasonable 
to use the same notation for sets and spaces of maps, since the latter are just given by  prescribing
a topology on the given sets.  However, the notational distinction is vital in the 
general context we are headed towards. Formally, (\ref{Gfixed}) and (\ref{enrichset}) describe a double 
enrichment of the underlying category $G\sU$: it is enriched in spaces, but as a closed symmetric monoidal
category it is also enriched over itself, that is, it is enriched in $G$-spaces.
We shall be categorically precise when we generalize.

With the notations of the general theory in \cite{GM}, we take $\sD$ to be $\sO_{\sF}$,
the full $\sU$-subcategory of $G\sU$ 
whose objects are the orbit $G$-spaces $G/H$ with $H\in \sF$. 
The most important example is $\sF=\sA\ell\ell$, the set of all
subgroups of $G$, and we write $\sO_G$ for this orbit category.
We have the adjunction
\begin{equation}\label{orbadj}
 \xymatrix@1{ \mathbf{Pre} (\sO_{\sF},\sU)\ar@<.4ex>[r]^-\bT & G\sU. \ar@<.4ex>[l]^-\bU}\\ 
\end{equation}
For a $G$-space $Y$, $G\sU(G/H,Y)$ can be identified with the fixed 
point space $Y^H$, so that $\bU(Y)$ is the presheaf of fixed point
spaces $Y^H$ for $H\in \sF$. 
As is easily checked, the left adjoint 
$\bT$ takes a presheaf $X$ to the $G$-space $X(G/e)$.  Therefore 
it takes the represented presheaf $\bY(G/H)$ {(see \S\ref{TU})} to the $G$-space $G/H$. 
Clearly $\epz\colon \bT\bU\rtarr Id$ is a natural isomorphism, hence $\bU$ is full and faithful.

Define the $\sF$-equivalences in $G\sU$ to be the $G$-maps $f$ 
such that the fixed point map $f^H$ is a weak equivalence for $H\in \sF$. 
These are the weak $G$-equivalences when $\sF = \sA\ell\ell$, and the $G$-Whitehead
theorem says that a weak $G$-equivalence between $G$-CW complexes is a $G$-homotopy 
equivalence. When $\sF = \{e\}$, a weak $\sF$-equivalence is just a $G$-map which 
is a nonequivariant weak equivalence, giving the naive or Borel version of equivariant homotopy theory.
Similarly, define the $\sF$-fibrations to be the $G$-maps $f$ such that $f^H$ is a (Serre) 
fibration for $H\in\sF$.  The following three theorems are proven in  \cite{MM, EHCT, Pia}. 
As we shall indicate, they are special cases of general results in \cite{GM}.  Compactly generated 
model categories are discussed in \cite{GM, MP}.

\begin{thm}\mylabel{ein} {The category} $G\sU$ is a compactly generated proper $\sU$-model category
with respect to the $\sF$-equivalences, $\sF$-fibrations, and the resulting cofibrations.
The sets of maps $\mathcal{I}_{\sF}=\{G/H\times i\}$ and $\mathcal{J}_{\sF}=\{G/H\times j\}$, where 
$H\in\sF$, $i\in \mathcal{I}$, and $j\in\mathcal{J}$, are generating sets of cofibrations and 
acyclic cofibrations.
\end{thm}

For $H\in \sF$, we have the functor $F_{G/H}\colon G\sU\rtarr \mathbf{Pre} (\sO_{\sF},\sU)$ from
\cite[1.5]{GM}; for a $G$-space $Y$, it is given by $(F_{G/H}Y)(G/K) = \sO_{\sF}(G/K,G/H)\times Y$. 
If $G$ acts trivially on $Y$, this is 
\[ (G/H) ^K\times Y\iso (G/H\times Y)^K = \bU(G/H\times Y)(G/K), \]
and then
\[ F_{G/H}Y \iso \bU(G/H\times Y).\] 
We apply this identification with $Y$ replaced by the maps in $\mathcal{I}$ and $\mathcal{J}$.

\begin{thm}\mylabel{zwei} {The category} $\mathbf{Pre} (\sO_{\sF},\sU)$ is a compactly generated proper $\sV$-model 
category with respect to the {level equivalences, level fibrations}, and the resulting cofibrations.
The sets of maps $F_{G/H}i$ and $F_{G/H}j$, where $H\in\sF$, $i\in \mathcal{I}$, and $j\in\mathcal{J}$,
generate the cofibrations and acyclic cofibrations, and these sets are isomorphic to
$\bU\mathcal{I}_{\sF}$ and $\bU \mathcal{J}_{\sF}$.
\end{thm}

\begin{thm}\mylabel{drei} {The pair} $(\bT,\bU)$ is a Quillen $\sU$-equivalence between {the categories} $G\sU$ and 
$\mathbf{Pre} (\sO_{\sF},\sU)$.
\end{thm}

{In light of \myref{drei}, we refer to the weak equivalences and fibrations of the presheaf category $\mathbf{Pre} (\sO_{\sF},\sU)$ as level $\sF$-equivalences and level $\sF$-fibrations.}

\myref{zwei} holds by \cite[4.30]{GM}.  The only point requiring verification
is that, in the language of \cite[4.13]{GM}, $\mathcal{J}_{\sF}$ satisfies the acyclicity 
condition for the level $\sF$-equivalences.  This means that any relative cell complex 
$A\rtarr X$ constructed from $\bU\mathcal{J}_{\sF}$ is a level $\sF$-equivalence.
By definition, $\bU$ creates  the $\sF$-equivalences and $\sF$-fibrations in $G\sU$, 
as in \cite[1.16]{GM}, and \cite[1.17]{GM} applies to prove Theorems \ref{ein} and \ref{drei}.
In this example, its acyclicity condition means that any relative $\mathcal{J}_{\sF}$-cell 
complex is an $\sF$-equivalence.  

As observed in \cite[p. 40]{MM}, $\bU$ carries $\mathcal{I}_{\sF}$-cell complexes and $\mathcal{J}_{\sF}$-cell 
complexes in $G\sU$ to  corresponding cell complexes in  $\mathbf{Pre} (\sO_{\sF},\sU)$, mapping relative
cell complexes with source $Y$ bijectively to relative cell complexes with source $\bU Y$.  This is not formal
but rather depends on the fact that $\bU$ preserves certain pushouts \cite[III.1.10]{MM}.  Therefore the acyclicity 
condition needed to prove Theorem \ref{ein} follows from that needed to prove \ref{zwei}.  However, since the maps
in $\mathcal{J}$ are inclusions of deformation retracts, the maps in $\mathcal{J}_{\sF}$ are inclusions of $G$-deformation 
retracts, hence the maps in $\bU \mathcal{J}_{\sF}$ are levelwise deformation retracts.  The acyclicity of relative 
$\bU \mathcal{J}_{\sF}$-cell complexes follows by passage to coproducts, pushouts, and sequential colimits.

To prove \myref{drei}, we observe that, for a $G$-space $Y$ and a space $V$, the maps $\et$ of 
\cite[1.12]{GM} are the evident homeomorphisms 
$Y^H\times V \iso (Y\times V)^H$. This implies that $\et\colon X\rtarr \bU\bT X$
is an isomorphism when $X$ is the domain or codomain of a map in $\mathcal{I}_{\sF}$.  Again using that
$\bU$ preserves the relevant colimits, it follows (as in \cite[1.20]{GM}) that 
$\et\colon X\rtarr \bU\bT X$ is an isomorphism in $\mathbf{Pre} (\sO_{\sF},\sU)$ for all 
cofibrant $X$.  

Using smash products instead of cartesian products and giving orbit $G$-spaces disjoint basepoints,
everything above works just as well using the categories $\sT$ and $G\sT$ of nondegenerately based spaces 
and nondegenerately based $G$-spaces instead of $\sU$ and $G\sU$. 

\section{Hopf groups and their actions}\label{Hopf}

\subsection{Hopf groups}\label{Hgrp}
We shall take the summary of the previous section as a template for generalization, and we must first generalize 
the notion of  a topological group.  As in the introduction, we take $\sV$ to be a cosmos, that is, a bicomplete 
closed symmetric monoidal category with product $\otimes$ and unit object $I$.  We fix $\sV$ throughout the paper.  We write $\sV(X,Y)$ for the
set of maps $X\rtarr Y$ in $\sV$, and we write $\ul{\sV}$ for the internal hom in $\sV$.  Then
\begin{equation}\label{nonequi}  
\sV(X,Y) = {\sV}(I, \ul{\sV}(X,Y)). 
\end{equation}
The $\sV$-functor $\ul{\sV}$ is characterized by the enriched adjunction 
$$   \ul{\sV}(X \otimes Y, Z) \iso \ul{\sV}(X,\ul{\sV}(Y,Z)).$$
Applying ${\sV}(I, -)$ we obtain the ordinary (set level) adjunction
$$  {\sV}(X \otimes Y, Z) \iso {\sV}(X,\ul{\sV}(Y,Z)).$$

As a starting point, we would like generalizations of fixed point $G$-spaces and generalizations of (\ref{Gfixed})  and (\ref{enrichset}) 
for group actions in a general cosmos $\sV$, namely 
\begin{equation}\label{Gfixed2}
 \ul{G\sV}(X,Y) =  \ul{\sV}_G(X,Y)^G  
\end{equation}
and therefore 
\begin{equation}\label{enrichset2}  
{G\sV}(X,Y) = {\sV}(I, \ul {G\sV}(X,Y)) = {\sV}(I, \ul{\sV}_G(X,Y)^G) = G\sV(I,\ul{\sV}_G(X,Y)). 
\end{equation}

Once we understand the category $G\sV$ of $G$-objects and $G$-maps as a cosmos, with internal hom objects $\ul{\sV}_G(X,Y)$
in $G\sV$, the agreement of the first and last terms in (\ref{enrichset2}) will be a special case of (\ref{nonequi}). There is no problem when $\sV$ is cartesian monoidal, so that 
$\otimes = \times$ and $I$ is a terminal object $\ast$ in $\sV$.  In that case, a $\sV$-group is just a group in $\sV$, defined via the 
usual diagrams, and (\ref{Gfixed2}) holds.   In general, we  must first define what we mean by a group 
object in $\sV$.  

 \begin{defn} A comonoid $C$ in $\sV$ is a monoid in the opposite category $\sV^{op}$.  Thus it is an object of $\sV$ 
 together with a {comultiplication} $\ps\colon  C\rtarr C\otimes C$  and an augmentation $\epz\colon C\rtarr I$ such that the 
 duals of the diagrams defining a monoid in $\sV$ commute. We say that $C$ is cocommutative if $\ga\ps = \ps$, where
 $\ga$ is the symmetry isomorphism in $\sV$. Let $\Co$ denote the category of cocommutative comonoids in $\sV$.  
 Note that $I$ is a cocommutative comonoid with $\ps$ the unit isomorphism and  $\epz$ the identity.  A unit for 
 a comonoid $C$ is a map of comonoids $\et\colon I\rtarr C$. 
  \end{defn}

The following elementary and quite standard categorical observations explain the relevance of $\Co$.
 
 \begin{lem}\mylabel{CartComonoids} If $\sV$ is cartesian monoidal, then the forgetful functor $\Co \rtarr  \sV$ is an isomorphism of categories.
 \end{lem}
 \begin{proof}  For an object $X$ of $\sV$, the diagonal map  $\DE$ and the {unique} map $\epz\colon X\rtarr \ast$ 
 specify the unique comonoid structure on $X$, and it is cocommutative.
 \end{proof}
  
 \begin{lem}  The bifunctor $\otimes$ in $\sV$ extends to a bifunctor on $\Co$.
 \end{lem}
 \begin{proof} 
 The {comultiplication} on $C\otimes D$ is the composite
\[ \xymatrix@1{
 C\otimes D \ar[r]^-{\ps\otimes \ps}  & C\otimes C \otimes D \times D  
 \ar[r]^-{\id\otimes \ga\otimes \id} & C\otimes D \otimes C \times D, \\} \]
 and it is cocommutative.   The augmentation is 
  \[ \xymatrix@1{ C\otimes D \ar[r]^-{\epz\otimes \epz} & I\otimes I \iso I.  \qedhere \\}\]
  \end{proof}
  
  \begin{lem} A comonoid $C$ in $\sV$ is cocommutative if and only if $\ps\colon C\rtarr C\otimes C$ is a morphism of comonoids.
  \end{lem}
  \begin{proof} The dual statement about monoids is standard.  Writing a diagrammatic proof of that
  and reversing all the arrows gives the proof for comonoids.  
  \end{proof}
  
 \begin{lem}\mylabel{CoVcart}  The category $\Co$ is cartesian monoidal under the product $\otimes$. The object $I$ is terminal in $\Co$,
 the cartesian diagonal map $\Delta \colon C \rtarr C \otimes C$ is the comultiplication on $C$, and the cartesian product of 
 two maps is their tensor product.
 \end{lem} 
 \begin{proof} 
The coordinate projections are
$$\xymatrix@1{ C \iso C\otimes I & C\otimes D \ar[l]_-{\id \otimes \epz}\ar[r]^-{\epz\otimes \id}  & I\otimes D\iso D\\} $$
and the universal property is easily checked.
\end{proof}
 
More generally, $\Co$ also has pullbacks, but it  is not a cosmos since it is neither closed nor cocomplete.   Ignoring $\Co$, we have the notion of a monoid $G$ in $\sV$, given by 
a product $\ph$ and unit $\et$.  When $G$ is in $\Co$, with {comultiplication} $\ps$ and counit $\epz$, we have the notion of  a $\Co$-monoid, for which $\ph$ and $\et$ are required to be maps of comonoids.  For $\ph$, this means that
the following familiar diagram must commute, where $\ga$ is the symmetry isomorphism in $\sV$.
$$\xymatrix{ 
G\otimes G \ar[r]^-{\ph}  \ar[d]_{\ps\otimes \ps} & G \ar[r]^-{\ps} & G\otimes G \\
G\otimes G\otimes G \otimes G \ar[rr]_-{\id\otimes \ga\otimes \id} 
& & G\otimes G\otimes G\otimes G \ar[u]_{\ph\otimes \ph}\\} $$

Since $\Co$ is cartesian monoidal, we can define $\Co$-groups as well as $\Co$-monoids.  
A $\Co$-monoid $G$ is a $\Co$-group if there is map $\ch\colon G\rtarr G$  in 
$\sV$, which we call an antipode (as is standard for Hopf algebras), such that the following 
diagrams commute.
 \[ \xymatrix{
 G\otimes G  \ar@[2ex][rr]^-{\chi \otimes \id} & & G\otimes G \ar[d]^{\ph} \\
 G\ar[u]^{\ps}  \ar[r]_{\epz} &  I  \ar[r]_{\et}  & G\\} 
\ \ \text{and} \ \ \xymatrix{
 G\otimes G  \ar@[2ex][rr]^-{\id \otimes \chi} & & G\otimes G \ar[d]^{\ph} \\
 G\ar[u]^{\ps}  \ar[r]_{\epz} &  I  \ar[r]_{\et}  & G\\} 
 \]
The following result is proven exactly as in the special case of cocommutative Hopf algebras (e.g. \cite[\S21.3]{MP}).
 
 \begin{lem} \mylabel{analogs} The antipode $\chi$ is unique if it exists; it is an antihomomorphism, that is, a morphism of $\sV$-groups 
 $G\rtarr G^{op}$; and it is an involution, $\chi^2 = \id$.
 \end{lem}
 
 \begin{defn} A {\em Hopf group} in $\sV$ is a $\Co$-group.
 \end{defn}
 
 \begin{rem}\mylabel{discrete}  We map the category of sets into $\sV$ via the functor $I[-]$ that sends a set $S$ to the coproduct of 
 copies of $I$ indexed on $S$.  This functor is left adjoint to the functor $\sV(I,-)$.  Regarding the category of sets as cartesian 
 monoidal, the functor $I[-]$ is strong symmetric monoidal.   If $G$ is a discrete group, then $I[G]$ is a Hopf 
 group in $\sV$.  We think of it as  the $\sV$-group ring of $G$.  
 \end{rem}
 
To sum up, a Hopf group in $\sV$ is a cocommutative $\sV$-comonoid $(G,\ps,\epz)$  with an extension of structure 
$(G,\ps,\epz,\ph,\et,\ch)$ satisfying the group axioms.  Ignoring $\ch$, the definition encodes asymmetrically the usual defining properties of bialgebras:
	\begin{itemize}
		\item{}$\et$ is a unit for $(G,\ps,\epz)$ or, equivalently, $\epz$ is a counit for $(G,\ph,\et)$, and
		\item{}$\ph$ is a map of  $\sV$-comonoids or, equivalently, $\ps$ is a map of $\sV$-monoids.
	\end{itemize}

When $\sV = \bf{Set} $ is the category of sets, $G$ is just a discrete group.  When $\sV = \sU$, $G$ is a 
topological group.   When $\sV = \bf{sSet}$ is the category of simplicial sets, $G$ is a simplicial group.  When $\sV = R$-$\bf{Mod}$, $G$ is a 
cocommutative Hopf algebra over $R$.  This puts these examples and many others on an equal 
categorical footing. As illustrated  in \myref{discrete}, any strong symmetric monoidal functor $\sV\rtarr \sW$
between cosmoi sends Hopf groups in $\sV$ to Hopf groups in $\sW$.

\begin{rem} In a cartesian monoidal category, such as $\Co$, a Hopf group is the same data as a group object (cf. \myref{CartComonoids}). Since the forgetful functor $\Co \rtarr \sV$ is strong monoidal, it follows that every Hopf group is the strong monoidal image of a cartesian group object. Conversely, strong monoidal functors preserve Hopf groups, and hence every such image of a cartesian group object is a Hopf group.
\end{rem}

\begin{rem}  We can generalize the definition of a Hopf group by allowing $\sV$ to be braided monoidal and by not requiring
cocommutativity.  The resulting objects are studied in categorical combinatorics, where they are called Hopf monoids; see 
for example \cite{AM}.  In that generality, Hopf monoids are not monoids in a cartesian monoidal category, even if we do
require the {comultiplication} to be  cocommutative.  As pointed out to us by Marcelo Aguiar, when $\sV$ is only braided monoidal, the product $C\otimes D$ 
of cocommutative comonoids need not be cocommutative.
\end{rem}

\subsection{Actions of Hopf groups on objects of $\sV$-categories}\label{subsec:Hopfactions}
 Fix a Hopf group $G$ in $\sV$ and let $\sM$ be a category enriched in $\sV$, or a $\sV$-category for short.  
 We write $\ul{\sM}(X,Y)$ for the object in $\sV$ of  morphisms $X\rtarr Y$ in $\sM$ and we assume that $\sM$ 
 is $\sV$-bicomplete.  The standard exposition is \cite{Kelly}, but we shall review briefly the summary in \cite[\S4.1]{GM}.   
 The reader may prefer to focus on $\sM=\sV$ but the generality is essential to the applications;  
if $\sV$  is a common enriching category, such as $\sU$, $R$-$\bf{Mod}$,  or $\bf{sSet}$, one does 
 not want to focus just  on $\sV$.  

Before describing the actions of a Hopf group $G$ in a $\sV$-category $\sM$, we review the generalization of the tensor-hom adjunction to $\sM$. Recall 
the $\otimes$-product of $\sV$-categories $\sM$ and $\sN$, which is a $\sV$-category with the same objects as $\sM\times\sN$.  For objects $X,X'\in \sM$ and $Y,Y' \in \sN$, 
\[(\ul{\sM\otimes \sN})\big((X,Y),(X',Y')\big) = \ul{\sM}(X,X')\otimes \ul{\sN}(Y,Y'),\]
with the obvious units and composition defined in terms of those of $\sM$ and $\sN$.  
 In addition to being bicomplete in the usual sense, $\sM$ has tensor and cotensor $\sV$-bifunctors
\[ \odot\colon \sM\otimes \sV\rtarr \sM \ \ \ \text{and} \ \ \  
F\colon \sV^{\text{op}}\otimes \sM\rtarr \sM\]
that take part in $\sV$-adjunctions
\begin{equation}\label{biten} 
\ul{\sM}(M\odot V, N)\iso \ul{\sV}(V,\ul{\sM}(M,N))\iso \ul{\sM}(M,F(V,N)).
\end{equation}
These imply ordinary unenriched adjunctions
\begin{equation}\label{biten2} 
{\sM}(X\odot V, Y)\iso {\sV}(V,\ul{\sM}(X,Y))\iso {\sM}(X,F(V,Y)).
\end{equation}
By the discussion in \cite[\S4.1]{GM}, we are free to take tensors in the opposite order,  $V\odot M$ instead of $M\odot V$. 
That is convenient for our purposes when we take $V= G$.  With that convention, we have transitivity isomorphisms
\begin{equation}\label{biten3} 
(V\otimes W)\odot X \iso V\odot (W\odot X)
\end{equation}
for $V,W\in \sV$ and $X\in \sM$.

We can define a left $G$-action on an object $X\in \sM$ in three equivalent ways. The first one is perhaps most 
familiar, but to make sense of it we must use (\ref{biten3}) with $V = W = G$.
	\begin{itemize}
		\item{}A map $G\odot X \rtarr X$ in $\sM$ such that the evident diagrams commute.
		\item{}A map $G \rtarr \ul{\sM}(X,X)$ of $\sV$-monoids.
		\item{}A $\sV$-functor $X\colon \ul{G}\rtarr \sM$ such that $X(\ast) = X$.
	\end{itemize}
Here $\ul{G}$ denotes $G$ regarded as a $\sV$-category with a single object $\ast$. 
From a categorical point of view, the last definition is particularly convenient since {it allows us to describe all standard equivariant constructions in terms of Kan extension.}

A fundamental reason for focusing on Hopf groups $G$ and not just $\sV$-monoids is that the enrichments we saw
when $\sM=\sV = \sU$ generalize.   For $G$-objects $X$ and $Y$ in $\sM$, we let $\ul{\sM}_G(X,Y)$ denote the object 
$\ul{\sM}(X,Y)$ of  $\sV$ with the conjugation action by $G$ given by the functor
\[\xymatrix@1{
\ul{G} \ar[r]^-{\ps} & \ul{G}\otimes \ul{G} \ar[r]^-{\ch\otimes \id} & \ul{G}^{op}\otimes \ul{G} \ar[r]^{X^{op}\otimes Y} & \sM^{op}\otimes \sM \ar[r]^-{\ul{\sM}} & \sV.\\} \]
Then $\ul{\sM}_G(X,Y)$ is a bifunctor from $G$-objects in $\sM$ to $G$-objects in $\sV$.

Analogously, for $G$-objects $V\in \sV$ and $Y\in \sM$,  we write $F_G(V,Y)$ for the object $F(V,Y)$ of $\sM$ with the conjugation action given by the functor
\[\xymatrix@1{
\ul{G} \ar[r]^-{\ps} & \ul{G}\otimes \ul{G} \ar[r]^-{\ch\otimes \id} & \ul{G}^{op}\otimes \ul{G} \ar[r]^{V^{op}\otimes Y} & \sV^{op}\otimes \sM \ar[r]^-{F} & \sM.\\} \]
This is again a bifunctor to $G$-objects in $\sM$. 

If $V$ and $W$ are $G$-objects in $\sV$, then so is $V\otimes W$; the action is given by the $\sV$-functor
\[ \xymatrix@1{ \ul G \ar[r]^-\psi & \ul G\otimes \ul G \ar[r]^-{V\otimes W} & \sV \otimes \sV \ar[r]^-{\otimes} & \sV. \\} \]
This gives a bifunctor to $G$-objects in $\sV$.\footnote{For notational consistency, we might use some such notation
as $V\otimes^G W$, but we desist.}
For $G$-objects $Z$ in $\sV$ we have the natural isomorphism 
\[ \ul{\sV}_G(V\otimes W, Z) \iso \ul{\sV}_G(V, \ul{\sV}_G(W,Z))\]
of $G$-objects in $\sV$.
More generally, if $V$ is a $G$-object in $\sV$ and $X$ is a $G$-object in $\sM$, then $V\odot X$ is a $G$-object in $\sM$
and, for $G$-objects $Y$ in $\sM$ we have the natural isomorphisms of $G$-objects in $\sV$
\begin{equation}\label{tencoten}
 \ul{\sM}_G(V\odot X, Y) \iso \ul{\sV}_G(V, \ul{\sM}_G(X,Y)) \iso \ul{\sM}_G(X,F_G(V,Y)).
\end{equation}

\subsection{Categories of $G$-objects and functors relating them}\label{Cats}

Retaining the assumptions on $\sV$ and $\sM$ from the first paragraphs of \S\ref{Hgrp} and \S\ref{subsec:Hopfactions}, define $G\sM$ to be the category of $G$-objects in $\sM$ and $G$-maps between them.  
In particular, we have the category $G\sV$.  Remember that $\sV$ is a cosmos; that is, a bicomplete closed symmetric monoidal category
with product $\otimes$, unit object $I$, and internal hom $\ul{\sV}$, and that $\sM$ is a bicomplete $\sV$-category.  Our equivariant categories have double enrichment, in both $\sV$ and $G\sV$, just as we saw in the case of $G\sU$.   

First, $G\sM$ is enriched in $\sV$.  For $G$-objects $X$ and $Y$ in $\sM$, we define the object $\ul{G\sM}(X,Y)$ in $\sV$ to be the equalizer of two maps 
\[\la,\rh\colon \xymatrix@1{ \ul{\sM}(X,Y) \ar@<.4ex>[r] \ar@<-.4ex>[r] & \ul{\sV}(G,\ul{\sM}(X,Y)) }\]
Thinking of (left) actions as functors $\ul{G} \rtarr \sM$, 
$\la$ and $\rho$ are the adjoints of 
\[\xymatrix{ \ul{G} \ar[r]^-Y & \sM \ar[rr]^-{\ul{\sM}(X,-)} & & \ul{\sV} \\}	\quad	\text{and}	\quad \xymatrix{\ul{G}^{op} \ar[r]^-{X^{op}}  & \sM^{op} \ar[rr]^-{\ul{\sM}(-,Y)} & & \ul{\sV}, \\} \]
respectively. 
Unravelling this, one checks that it is an enriched encapsulation of the desired equivariance relation $f(gx) = gf(x)$.

For an object $X\in \sM$, we let $\epz^* X$ denote $X$ with the trivial $G$-action
$$ \xymatrix{ G\odot X \ar[r]^-{\epz\odot \id} & I \odot X \iso M.\\}  $$
In particular we agree to regard $I$ as the $G$-trivial $G$-object $\epz^*I$ in $\sV$.
We have already defined $\ul{\sM}_G(X,Y)$ to be the object of $G\sV$ obtained by giving 
$\ul{\sM}(X,Y)\in\sV$ its conjugation $G$-action, and we define 
$$  \ul{\sM}_G(X,Y)^G = \ul{G\sV}(I, \ul{G\sM}(X,Y)). $$
Now $\ul{\sV}_G(V,W)$ gives the internal hom required to make sense of the following crucial, but
elementary, result.
\begin{thm}  For any Hopf group $G$, the category $G\sV$ is a cosmos.  It has double enrichment, in $\sV$ and $G\sV$,
related by a canonical natural isomorphism
$$  \ul{\sV}_G(V,W)^G \iso \ul{G\sV}(V,W).$$
\end{thm}

The product is $\otimes$ with diagonal $G$-action, the unit is $I$ with trivial $G$-action, the internal hom
is $\ul{\sV}_G$ and the hom in $\sV$ is $\ul{G\sV}$. The isomorphism is a comparison of
equalizers.  Limits and colimits are constructed in $\sV$  and given $G$-actions induced by the actions on inputs.  
Similarly, we have the following result.

\begin{thm}  The category $G\sM$ is a bicomplete $G\sV$-category with double enrichment in $G\sV$ and $\sV$
related by a canonical natural isomorphism 
$$  \ul{\sM}_G(X,Y)^G \iso \ul{G\sM}(X,Y).$$
\end{thm}

The tensors are $V\odot X$ with diagonal $G$-action and the cotensors are $F_G(V,X)$.
We shall discuss the double enrichment, in $\sV$ and $G\sV$, a little more categorically
and say a bit about the proofs in \S\ref{double}.

The essential starting point for enriched equivariant homotopy theory is 
an understanding of the fixed point objects $X^H$ and orbit objects $H \backslash X$
in $\sM$ for objects $X\in G\sM$ and subgroups $H$ of $G$.{\footnote{In \cite{DK}, 
the authors start with a simplicially enriched category $\sN$ and a set $\sO$ 
of objects, which they call `orbits', in $\sN$. For $O\in\sO$ and $N\in\sN$, 
they view the simplicial sets $\sN(O,N)$ as analogues of fixed point 
objects. When $\sN=G\t{-}\mathbf{sSet}$, their context leads to the simplicial analogue of
\S\ref{GTop}.  However, their general context is not relevant to the equivariant theory 
discussed here since the natural fixed point objects $N^H$ are in $\sN$ and not $\mathbf{sSet}$, so play
no role in their theory.}  We also need induction and coinduction functors $H\sM\rtarr G\sM$ 
for inclusions $\io\colon H\subset G$.  Just as for spaces, inclusions are understood to be ``closed''.  
We give an appropriate categorical meaning of closed in \myref{Defnclosed}. 

If we view $G$ as a $\sV$-category with a single object, the needed functors can be specified as suitable 
(weighted) limits and colimits induced from change of group homomorphisms, by which we understand 
morphisms of Hopf groups. In fact, they are all left or right Kan extensions along obvious change of group
functors induced by group homomorphisms.  However, we want a more concrete categorical perspective, 
and we leave it to the categorically minded reader to check that the definitions we give are indeed the Kan
extensions we indicate.

Let $\io\colon H\subset G$ be an inclusion.  We first define  ``orbit tensors'' $V\odot_H X$ 
for left $H$-objects $X\in \sM$ and right $H$-objects $V\in\sV$ and 
``fixed point cotensors'' $HF(V,X)$ for left $H$-objects $X\in \sM$ 
and $V\in \sV$.  These are objects of $\sM$, and they specialize to give 
change of group functors that are entirely analogous to those in familiar examples.

\begin{defn}\mylabel{Hcoten} Let $V$ be a right $H$-object in {$\sV$} and $X$ be a left $H$-object in $\sM$. 
Using the associativity isomorphism (\ref{biten3})  implicitly, 
define $V\odot_H X$ in $\sM$ to be the coequalizer   
{\[ \xymatrix@1{V\otimes H \odot X \ar@<.4ex>[r] \ar@<-.4ex>[r] & V\odot X \ar[r] & V\odot_H X.\\}\]}
Dually, for left $H$-objects $V$ in $\sV$ and $X$ in $\sM$, define $HF(V,X)$ in $\sM$ to be the equalizer
\[\xymatrix@1{HF(V,X) \ar[r] & F(V,X) \ar@<.4ex>[r] \ar@<-.4ex>[r] & F(H\otimes V,X).\\}\]
One of each of the parallel pairs of arrows is induced by the action of $H$ on $V$ and the other is 
induced by the action of $H$ on $X$. 
\end{defn}

We have the obvious universal properties, given in \myref{OrbitAdjoint} below.  If we specialize by taking $V=G$, then the left action of $G$ on
$G$ induces a left action on $G\odot_H X$ and the right action of $G$ on $G$ induces a left action of $G$
on $HF(G,X)$.  Composition of $\io\colon \ul H \rtarr \ul G$ with actions $\ul G\rtarr \sM$ gives the restriction 
$\sV$-functor $\io^*\colon \ul{G\sM}\rtarr \ul{H\sM}$. \myref{Hcoten} gives explicit identifications of the left and right Kan
extensions along $\io$, hence we have the following enriched adjunctions.

\begin{lem}\mylabel{OrbitAdjoint}  There are $\sV$-adjunctions 
\[ \ul{G\sM}(G\odot_{H}Y,X)\iso \ul{H\sM}(Y,\io^*X) \quad \text{and} \quad \ul{H\sM}(\io^*X,Y)\iso \ul{G\sM}(X, HF(G,Y)),\]
where $X\in G\sM$ and $Y\in H\sM$.
\end{lem}

Thinking of $\io^*$ as a restriction functor $\bR^G_H$,  we can think of {$G\odot_H-$} 
and $HF(G,-)$ as induction and coinduction.   In particular, applying this to $\et\colon I\rtarr G$, we obtain
free and cofree $H$-objects in $\sM$.

Just as for $\sV$, we write $\epz^*X$ for an object $X$ of $\sM$ with the trivial action of $G$, which is induced 
by applying $-\odot X$ to the counit $\epz\colon G\rtarr I$. 

\begin{defn} For $Y\in H\sM$, such as $Y=\io^*X$ for $X\in G\sM$,  define the orbit objects 
$H\backslash Y$ and fixed point objects $Y^H$ in $\sM$ to be 
\[ H\backslash Y  = \epz^* I\odot_{H} Y \ \ \ \text{and} \ \ \ Y^H= HF(\epz^* I, Y).\]
Specializing to $\sM=\sV$ and using the right action of $G$ on itself rather than the left action,
define orbit $G$-objects in $\sV$ by
 \[G/H = G\otimes_H \epz^*I .\]
\end{defn}

These functors can be identified as the left and right Kan extensions along $\varepsilon$, hence we have the 
following enriched adjunctions.

\begin{lem} There are $\sV$-adjunctions 
\[ \ul{H\sM}(Y,\epz^*Z)\iso \ul{\sM}(H\backslash Y,Z) \quad \text{and} \quad \ul{H\sM}(\epz^*Z, Y)\iso \ul{\sM}(Z,Y^H), \]
where $Y\in H\sM$ and $Z\in\sM$. 
\end{lem}

\begin{rem}\mylabel{InitFixedPts}  When $Y = \varnothing$ is an initial object of $H\sM$ (with trivial $H$-action), $H\backslash Y$ and $Y^H$ are also initial objects.   For the left adjoint, $H\backslash Y$, this is automatic.  For the right adjoint, $Y^H$, it holds because the equalizer
$\varnothing^H \rtarr \varnothing$ is a monomorphism with a section (because $\varnothing$ is initial), hence $\varnothing^H \cong \varnothing$.
\end{rem}

Just as in familiar examples, for $X\in G\sM$ we have a natural isomorphism 
\begin{equation} 
\xymatrix@1{G \odot_H\io^*X \ar[r]^{\iso} &  G/H\odot X \\}
\end{equation}
in $G\sM$, where the diagonal $G$-action is used on the right. It is the adjoint of the 
$H$-map $\io^*X\rtarr \io_*(G/H\odot X)$  induced by the canonical $H$-map $I = H/H\rtarr G/H$,
and its inverse is the composite
$$\xymatrix@1{ G\odot X \ar[r]^-{\DE\odot{\id}} & (G\otimes G) \odot X \iso G\odot(G\odot X) \ar[rr]^-{\id\odot \al(\ch\odot\ \id)} 
& & G\odot X\ar[r]^-{\pi} & G\odot_H \io^* X, \\} $$
where $\al$ is the action of $G$ on $X$ and $\pi$ is the canonical map. Composing adjunctions and omitting $\io^*$ from the notation,  for $Z\in \sM$ we obtain
\begin{equation}\label{keyG}
 \ul{G\sM}(G/H\odot \epz^*Z ,X)\iso \ul{\sM}(Z, X^H).
\end{equation}
Specializing to $\sM = \sV$ and  $Z=I$, (\ref{keyG}) gives
\[\ul{G\sV}(G/H,V)\iso V^H\]
for $V\in \sV$.
A further comparison of definitions gives the usual identifications
\begin{equation}\label{keyGthree}
H\backslash X \iso H\backslash G\odot_G X \ \ \text{and}\ \ X^H\iso GF(G/H, X).
\end{equation}

\section{Equivariant enriched model and presheaf categories}\label{equiv2}

\subsection{Standing assumptions and technical hypotheses}
We turn to model category theory.  As in the previous section, we fix a Hopf group $G$ in a cosmos $\sV$ and a bicomplete $\sV$-category $\sM$. 
Subgroups of $G$ are understood to be Hopf subgroups that are closed in the sense defined in \S\ref{closed}.

As in \cite[\S1.1]{GM}, we now add in standing model theoretic assumptions.  To avoid interrupting exposition with technicalities later, we then
give supplements.   In particular, as discussed generally in \cite{GM}, the assumption that $I$ is cofibrant can be weakened  
as in  \myref{unitbit},  but we find it convenient to assume it as the default.
	\begin{itemize}
  \item{}  We assume that $\sV$ is a cofibrantly generated proper monoidal model category. 
  
  \item{}  We assume that the unit $I$ of $\sV$ is cofibrant. 
  
\item{} We assume that $\sM$ is a cofibrantly generated $\sV$-model category with generating cofibrations $\mathcal{I}_\sM$ 
and generating acyclic cofibrations $\mathcal{J}_\sM$.
	\end{itemize}

Of course, we can take $\sM = \sV$.  We can place discrete groups $G$ in this general context by specializing to the $\sV$-Hopf group $I[G]$ and here we certainly want $I$ to be cofibrant in $\sV$; compare \cite[\S4.5]{GM}.  

We work with a general set $\sF$ of subgroups of $G$. We define the appropriate notion of a family $\sF$ in \S\ref{families}, but we shall not restrict attention to families here.  
We always assume that the trivial subgroup $e$ is in $\sF$; we can think of it as the unit $\et\colon I\rtarr G$.

\begin{rem} For a discrete group $G$, we have the family of (closed) subgroups $H$ of $I[G]$, as defined in \S\ref{closed}, 
and the set of those subgroups of the form $I[H]$, where $H$ is an ordinary subgroup of $G$. The latter is not a family in general, and our theory applies to both. 
\end{rem}

As in \cite{GM} properness will play little role in this paper, but the following notions will be needed to prove that left properness is inherited equivariantly, and it 
can also be used to prove that $\sV$-model structures exist when the unit of $\sV$ is not cofibrant.

\begin{defn}\mylabel{MGood} The set $\sF$ is \emph{$\sM$-good} if the functors $(G/H)^K \odot (-) \colon \sM \rtarr \sM$ preserve 
cofibrations for any  subgroups $H,K \in \sF$.
The set $\sF$ is very $\sM$-good if, in addition, the functor $(-)^H$ commutes with the tensors, coproducts, pushouts, and directed colimits 
that appear in the construction of relative $\sF\mathcal{I}_\sM$-cell complexes.
\end{defn}

\begin{rem}  Since $\sM$ is a $\sV$-model category, $\sF$ is $\sM$-good if every $(G/H)^K$ is cofibrant in $\sV$.
\end{rem}

\begin{exmp} Consider $\sM = \sU$ and a topological group $G$.  
	\begin{itemize}
		\item[(a)]  If $G$ is discrete, then every $(G/H)^K$ is discrete
				and every $\sF$ is $\sU$-good.
		\item[(b)]  If $G$ is a compact Lie group and $K$ and $H$ are closed subgroups,  then $(G/H)^K$ is a closed submanifold of $G/H$.  Here again
		every $\sF$ is $\sU$-good. 
				\item[(c)]  For a general topological group $G$, $\sF$ is rarely $\sU$-good. Despite this fact, the categories $G\mathbf{Top}$ and $\mathbf{Pre}(\sO_\sF,\sU)$ are left proper in full generality. One proof plays the Hurewicz and Quillen model structures on $\sU$ off of each other (cf. \cite[Theorem 6.5]{MMSS}), but we have not tried to formalize these techniques.
	\end{itemize}
\end{exmp}

\begin{exmp} Let $G$ be a discrete group and $\sV$ be a cosmos.  {The natural set map	
\[  I[-] : G\mathbf{Set}(G/K,G/H) \rtarr (G\sV)(I[G/K],I[G/H])  \]
transposes to a $\sV$-map $I[G\mathbf{Set}(G/K,G/H)]\rtarr \ul{G\sV}(I[G/K],I[G/H]).$
Assume that this map is an isomorphism in $\sV$ for all $H,K\subset G$.  Then $(I[G]/I[H])^{I[K]} \cong I[(G/H)^K]$ is cofibrant (assuming that $I$ is cofibrant).}  Therefore a set $\sF$ of subgroups of $I[G]$ is $\sV$-good if all 
elements of $\sF$ are of the form $I[H]$ for a subgroup $H$ of $G$.
\end{exmp}

\subsection{Equivariant model categories and functor categories in $\sM$} \label{GMModel}

As in \S\ref{GTop}, we view $\sF = \sA\ell\ell$ as the most interesting example.  It leads to ``genuine''
equivariant homotopy theory.  The example $\sF =\{e\}$ leads to ``naive'' or ``Borel" equivariant homotopy theory.
Since these notions of genuine and naive are different from the ones now standard in
equivariant stable homotopy theory, we make little use of the terms, but they express our point of view.
We say that a $G$-map $f$ is a $G$-equivalence if each $f^H$ is a nonequivariant weak equivalence.
We might say that $f$ is an $e$-equivalence if $f^e$, the underlying nonequivariant map, is a weak equivalence.  
These extremes are the special cases $\sF= \sA\ell\ell$ and $\sF =\{e\}$ of the following definition.

\begin{defn}  A $G$-map $f\colon M\rtarr N$ between objects of $G\sM$
is an $\sF$-equivalence or $\sF$-fibration if $f^H\colon M^H\rtarr N^H$
is a weak equivalence or fibration in $\sM$ for all $H\in \sF$; $f$ is
an $\sF$-cofibration if it satisfies the LLP with respect to all acyclic
$\sF$-fi\-bra\-tions.  Define $\sF \mathcal{I}_{\sM}$ and $\sF \mathcal{J}_{\sM}$ to 
be the sets of maps obtained by applying the functors $G/H\odot (-)$ to the
maps in $\mathcal{I}_{\sM}$ and $\mathcal{J}_{\sM}$, where $H\in\sF$.  
\end{defn}

Specializing \cite[4.16]{GM}, we obtain the following result. 

\begin{thm}\mylabel{GM}  If the sets $\sF \mathcal{I}_{\sM}$ and $\sF \mathcal{J}_{\sM}$ admit the 
small object argument and $\sF \mathcal{J}_{\sM}$ satisfies the acyclicity condition for the
$\sF$-equivalences, then $G\sM$ is a cofibrantly generated $\sV$-model category
with generating cofibrations $\sF \mathcal{I}_{\sM}$ and acyclic cofibrations  
$\sF \mathcal{J}_{\sM}$. If $\sM$ is right proper, then so is $G\sM$.  If $\sF$ is very $\sM$-good and $\sM$ 
is left proper, then so is $G\sM$.
\end{thm}

Here we have omitted condition (ii) of \cite[4.16]{GM} since it follows formally from (\ref{keyG}),
and (\ref{keyG}) also reduces the small object argument to a question
about colimits of (transfinite) sequences in $\sM$ that are obtained by 
passing to $H$-fixed points from relative cell complexes in $G\sM$.  In
practice, the small object argument always applies.  

The acyclicity condition, which is condition (i)  of \cite[4.16]{GM}, will 
hold provided that passage to $H$-fixed points from a relative 
$\sF \mathcal{J}_{\sM}$-cell complex gives a weak equivalence in $\sM$ for each $H\in \sF$.
In the topological situation of \S\ref{GTop}, the essential point is that passage
to $H$-fixed points commutes with pushouts, one leg of which is a closed inclusion.
This implies that the fixed point presheaf of an $\sF\mathcal{J}_\sM$-cell complex is an acyclic cell complex.

To show that $G\sM$ is a $\sV$-model category, we must show that for every cofibration $i\colon A\rtarr X$ and fibration 
$p\colon E\rtarr B$ in $G\sM$, the map 
\[ \ul{G\sM}(i^*,p_*):\ul{G\sM}(X,E)\rtarr \ul{G\sM}(A,E)\times_{\ul{G\sM}(A,B)} \ul{G\sM}(X,B)\]
is a fibration and is an $\sF$-equivalence if either $i$ or $p$
is an $\sF$-equivalence. It is enough to consider the case of a generating cofibration 
$i\colon G/H\odot M\rtarr G/H\odot N$, where $H\in \sF$. The map in question
then takes the form
\[ \xymatrix{
\ul{G\sM}( G/H\odot N,E) \ar[d] \\  
\ul{G\sM}( G/H\odot M,E)\times_{\ul{G\sM}(G/H\odot M,B)} \ul{G\sM}( G/H\odot N,B).\\}  \]
This map is isomorphic to the map
\[\ul{\sM}(N,E^H) \rtarr \ul{\sM}(M,E^H)\times_{\ul{\sM}(M,B^H)} \ul{\sM}(N,B^H).\]
The conclusion holds since $E^H\rtarr B^H$ is a fibration and $\sM$ is a $\sV$-model category.

Since $\sF$-fibrations and $\sF$-equivalences are determined on fixed points, and $(-)^H$ preserves pullbacks $G\sM$ automatically inherits right properness from $\sM$.
When $\sF$ is very $\sM$-good, one can check that $(-)^H$ preserves cofibrations and preserves pushouts one leg of which is a cofibration, hence $G\sM$ also inherits left
properness from $\sM$.

We can compare the model structures on $G\sM$ of \myref{GM} to model categories 
of presheaves in $\sM$, generalizing \myref{drei}.   We have the $\sV$-category $\mathbf{Fun} (\ul{G},\sM)$ of 
$\sV$-functors $\ul{G}\rtarr \sM$, alias $G$-objects in $\sM$.
The underlying category of $\mathbf{Fun} (\ul{G},\sM)$ is $G\sM$.
Since $\ul{G} = \sO^{op}_{e}$, this puts us in the case $\sF=\{e\}$. Evaluation at the single object of our domain category forgets the $G$-action,
and its left adjoint, $F_{G/e}$, sends an object $M\in \sM$ to the free $G$-object $G\odot M$ in $G\sM$.  
Here the level $\sV$-model structure of \cite[4.30]{GM} coincides  with the $\{e\}$-model structure on $G\sM$ of 
\myref{GM}. We regard this as a naively equivariant model structure, rather than a truly equivariant one.  

For larger sets $\sF$, such as $\sA\ell\ell$, we replace $\ul{G}$ by the orbit category $\sO_{\sF}$.

\begin{defn}
The orbit category $\sO_G$   of $G$ is the full $\sV$-subcategory of {$G\sV$} whose objects are the orbits $G/H$.
{Given a set $\sF$ of subgroups of $G$, the $\sV$-category}
 $\sO_{\sF}$ is the full {$\sV$-subcategory} of $\sO_G$ whose objects are the $G/H$ for $H\in \sF$.   
\end{defn}

\begin{rem}  We reiterate that when we specialize to $I[G]$ for a discrete group $G$
and a set $\sF$ of ordinary subgroups of $G$,  rather than Hopf subgroups of $I[G]$, the category $\sO_{\sF}$ is not to be confused with 
either the full $\sV$-category $I[\sO_{\sF}]$ with objects the $I[G/H]$ or its possibly non-full $\sV$-subcategory with
morphism objects
\[ I[\sO_{\sF}](I[G/H],I[G/K]) = I[(G/K)^H]. \]  
We shall restrict attention to full subcategories of $\sO_G$ for simplicity.  Much of the theory generalizes, but we won't usually get Quillen equivalences 
as in \myref{drei} in the non-full case.  When $\sV$ is cartesian monoidal and $G$ is a Hopf group in $\sV$, that is, a group internal to $\sV$, we
have no such distinctions: the full subcategory $\sO_{\sF}$ is generally the only variant in sight. 
\end{rem}

We described the $\sF$-model structure on $G\sM$ in \myref{GM}.  
For comparison, using the level $\sF$-classes of weak equivalences
and fibrations as in \cite[4.29]{GM}, \cite[4.30]{GM} gives a level $\sF$-model 
structure on $\mathbf{Fun} (\sO_{\sF}^{op},\sM)$.  Let $F_{G/H}$ denote
the presheaf in $\sV$ represented by the object $G/H$.  Its value on $G/K$ is   
\begin{equation}\label{wellohso}
\ul{G\sV}(G/K,G/H) \iso (G/H)^K. 
\end{equation}
As observed in \cite[\S5.1]{GM}, for a (small) $\sV$-category $\sD$, a presheaf $X$ 
in $\mathbf{Pre}(\sD,\sV)$, and an object $M\in \sM$, application of $\odot$ levelwise gives a functor
$X\odot M$ in  $\mathbf{Fun} (\sD^{op},\sM)$, and this construction is functorial. 
  
\begin{defn}\mylabel{FunFun}  Let $F_{\sF}\mathcal{I}_{\sM}$ and $F_{\sF}\mathcal{J}_{\sM}$ denote the 
sets of presheaves $F_{G/H}\odot i$ and $F_{G/H}\odot j$ in $\mathbf{Fun} (\sO_{\sF}^{op},\sM)$, where 
$H\in \sF$, $i\in \mathcal{I}_{\sM}$, and $j\in \mathcal{J}_{\sM}$.
\end{defn}

\begin{thm}\mylabel{GPM} If the sets $F_{\sF}\mathcal{I}_{\sM}$ and $F_{\sF}\mathcal{J}_{\sM}$ 
admit the small object argument and every relative $F_{\sF} \mathcal{J}_{\sM}$-cell complex is a
level $\sF$-equivalence, then $\mathbf{Fun} (\sO_{\sF}^{op},\sM)$ 
is a cofibrantly generated $\sV$-model category with generating cofibrations 
$F_{\sF}\mathcal{I}_{\sM}$ and acyclic cofibrations $F_{\sF}\mathcal{J}_{\sM}$. If $\sM$ is right proper, then so is $\mathbf{Fun}(\sO^{op}_\sF,\sM)$. 
If $\sF$ is $\sM$-good and $\sM$ is left proper, then so is $\mathbf{Fun}(\sO^{op}_\sF,\sM)$.
\end{thm}

The small argument condition is generally inherited from $\sM$, 
often reducing to a compactness observation in contexts of compactly 
generated model categories. In the cartesian monoidal case, the acyclicity condition 
is often an elaboration of the simple argument that  applied to spaces in \S\ref{GTop}.
The verification of the $\sV$-model category structure is similar to that given in \myref{GM}. 
We must show that the map $\ul{\mathbf{Fun}}(\sO_\sF^{op},\sM)(i,p)$ of \cite[4.19]{GM}
is a fibration and is acyclic if $i$ or $p$ is so, where $i\colon F_{G/H}\odot M\rtarr F_{G/H}\odot N$ 
is a generating cofibration and $p\colon E\rtarr B$ is a fibration. The map in question is isomorphic to
\[\ul{\sM}(N,E(G/H)) \rtarr \ul{\sM}(M,E(G/H))\times_{\ul{\sM}(M,B(G/H))}\ul{\sM}(N,B(G/H)).\]
The conclusion holds since $p\colon E(G/H)\rtarr B(G/H)$ is a fibration and $\sM$ is a $\sV$-model category. 
The inheritance of right and left properness is proven in the same way as for \myref{GM}.

\begin{rem}\mylabel{unitbit} If $I$ is not cofibrant, it is natural to assume that there is a cofibrant replacement 
$q\colon QI\rtarr I$ such that $q\odot \id\colon  QI\odot X\rtarr I\odot X\iso X$ is an $\sF$-equivalence for every 
cofibrant $X\in \sM$ and the functor $(-)^H$ commutes with tensoring with $q$ for each $H\in \sF$.  Then
$\mathbf{Fun}(\sO^{op}_\sF,\sM)$ inherits a $\sV$-model structure from $\sM$ if $\sF$ is good, and 
$G\sM$ inherits a $\sV$-model structure if $\sF$ is very good.
\end{rem}

Assuming the hypotheses of Theorems \ref{GM} and \ref{GPM}, we have the 
$\sF$-model categories $G\sM$ and
$\mathbf{Fun} (\sO_{\sF}^{op},\sM)$.  

\begin{thm}\mylabel{GObjAsPrshvs} There is a Quillen $\sV$-adjunction 
\[ \xymatrix@1{ \mathbf{Fun} (\sO_{\sF}^{op},\sM)\ar@<.4ex>[r]^-\bT & G\sM \ar@<.4ex>[l]^-\bU}\\ \]
and it is a Quillen equivalence if the functors $(-)^H$ preserve the
tensors, coproducts, pushouts, and directed colimits that appear 
in the construction of $\sF\mathcal{I}_\sM$-cell complexes.
\end{thm}

\begin{proof}  We have displayed the adjunction on underlying categories;
on the enriched level, the corresponding adjunction is a comparison of equalizer
diagrams. We shall elaborate a bit in \S\ref{TU}. For $N\in G\sM$, $\bU(N)_{G/H} = N^H$.
For $X\in \mathbf{Fun} (\sO_{\sF}^{op},\sM)$, $\bT X = X_{G/e}$.  
Since $\bU$ creates the $\sF$-equivalences and $\sF$-fibrations in $G\sM$, {the pair}
$(\bT,\bU)$ is a Quillen adjunction, and it is a Quillen equivalence if
and only if $\et\colon X\rtarr \bU\bT X$ is a level equivalence
when $X\in \mathbf{Fun} (\sO_{\sF}^{op},\sM)$ is cofibrant.  First consider
$X= F_{G/H}\odot M$, where $M\in \sM$ (not $G\sM$). Evaluated at $G/e$, this
gives $G/H\odot M$, by (\ref{wellohso}). Now take $K$-fixed points.
The assumption that $(-)^K$ preserves tensors means that the result is 
$(G/H)^K\odot M$. This agrees with $X_{G/K}$, and $\et$ is an isomorphism. This last sentence 
explains why we require $\sO_{\sF}$ to be a full subcategory of $G\sV$. The assumed commutation of passage to $K$-fixed points and the relevant 
colimits ensures that $\bU$ maps relative cell complexes to relative cell complexes 
bijectively and that $\et$ is an isomorphism for any cell complex 
$X$, just as for topological spaces in \S\ref{GTop}.
Note that this implicitly uses Remark~\ref{InitFixedPts} to begin the induction.
\end{proof}

\begin{rem} When $\sM$ is compactly generated \cite[\S15.2]{MP}, only sequential colimits (over $\omega$)
need be considered in the last statement.
\end{rem}

\subsection{Equivariant model categories and presheaf categories in $\sV$}\label{GVModel}

Now that we understand equivariant model categories as functor
categories in $\sM$, we can understand them as 
presheaf categories in $\sV$ whenever we can understand $\sM$
itself as a presheaf category in $\sV$.  That is, if we have an
answer to one of \cite[Questions 0.1 -- 0.4]{GM} for $\sM$, then we 
have an answer to an analogous question with $\sM$ replaced by $G\sM$.  
This is immediate from the standard observation that a 
functor category in a functor category is again a functor category.

\begin{prop}\mylabel{Composite} Let $\sD$ and $\sE$ be small $\sV$-categories and
let $\sN$ be any $\sV$-category.  Then there is a canonical
isomorphism of $\sV$-categories
\[ \mathbf{Fun} (\sD,\mathbf{Fun} (\sE,\sN)) \iso \mathbf{Fun} (\sD\otimes \sE,\sN).\]
If we have level $\sV$-model structures induced by a $\sV$-model structure on $\sN$ 
on all functor categories in sight, then this is an isomorphism of $\sV$-model categories.
\end{prop}

Now return to the equivariant context. Suppose that we have a $\sV$-model category 
$\sM$ with $\sV$-functorial factorizations together with a $\sV$-functor
$\de\colon \sD\rtarr \sM$ that gives rise to a Quillen equivalence {$\mathbf{Pre} (\sD,\sV) \xymatrix@1 { \ar@<.4ex>[r] & \ar@<.4ex>[l] } \sM $}.\footnote{Defining the adjunction
$(\bT,\bU)$ when $\de$ is not necessarily the inclusion of a full subcategory presents no difficulty.  See  \cite[\S 1]{GM} and \S\ref{TU} below.} Conditions 
ensuring this are discussed in \cite[\S1]{GM}, in answer to Questions 0.2 and 0.3 there.  Retaining the assumptions of the previous section,
for any family of subgroups $\sF$ we also have a Quillen equivalence 
{$\mathbf{Fun} (\sO_{\sF}^{op},\sM) \xymatrix@1 { \ar@<.4ex>[r] & \ar@<.4ex>[l] } G\sM $.}    These give a composite 
Quillen equivalence
{\[\xymatrix@1 {\mathbf{Fun} (\sO_{\sF}^{op},\mathbf{Pre} (\sD,\sV)) \ar@<.4ex>[r] & \ar@<.4ex>[l] \mathbf{Fun} (\sO_{\sF}^{op},\sM)  \ar@<.4ex>[r] & \ar@<.4ex>[l] G\sM  }   \]}
since the following lemma ensures that the {first pair} is a Quillen equivalence. 

\begin{lem} Suppose that $L \colon \sM \xymatrix@1 { \ar@<.4ex>[r] & \ar@<.4ex>[l] } \sN \colon R$ is a Quillen $\sV$-equivalence and that $\sM$ has $\sV$-functorial factorizations. Let $\sO$ be a small $\sV$-category, and suppose that $\mathbf{Fun}(\sO,\sM)$ and $\mathbf{Fun}(\sO,\sN)$ have projective $\sV$-model structures. Then the functors $L_* : \mathbf{Fun}(\sO,\sM) \xymatrix@1 { \ar@<.4ex>[r] & \ar@<.4ex>[l] } \mathbf{Fun}(\sO,\sN) : R_*$  induced by composition give a Quillen $\sV$-equivalence.
\end{lem}
\begin{proof} It is clear that $(L_* , R_*)$ is a Quillen $\sV$-adjunction. To show it is a
Quillen equivalence, it suffices to verify Quillen's condition  that for every cofibrant $X \in \mathbf{Fun}(\sO,\sM)$ and fibrant $Y \in \mathbf{Fun}(\sO,\sN)$, a map $f \colon  L_*X \rtarr Y$ is a weak equivalence if and only if its adjoint $\tilde{f} \colon X \rtarr  R_*Y$ is a weak equivalence.

It is often the case that the cofibrant objects in $\mathbf{Fun}(\sO,\sM)$ are levelwise cofibrant, and then the Quillen condition for $(L_* , R_*)$ follows by applying the Quillen condition for $(L , R)$ levelwise (as in \cite[11.6.5]{Hirschhorn}).  That argument applies when $\sF$ is $\sM$-good, but in fact essentially the same argument still works
in general.  By \cite[\S 45.1]{DHKS}, it suffices to check the Quillen condition on \emph{any} left $L_*$-deformation\footnote{For a functor $\Phi: \sM\rtarr \sN$ of homotopical categories, a left $\Phi$-deformation is an equivalence-preserving functor $r:\sM\rtarr \sM$ together with a natural weak equivalence $r\xrightarrow{\sim} \id_{\sM}$ such that $\Phi$ preserves those weak equivalences that are in the image of $r$.}
of the domain and right $R_*$-deformation of the codomain, 
and the objectwise cofibrant $X \in \mathbf{Fun}(\sO,\sM)$ form a left $L_*$-deformation because we have a $\sV$-functorial factorization on $\sM$. 
\end{proof}

\myref{Composite} allows us to rewrite this, giving the following general conclusion.

\begin{thm}\mylabel{neat}  When \myref{GObjAsPrshvs} applies to gives a Quillen equivalence between the 
 $\sF$-model category $G\sM$ and the functor category  $\mathbf{Fun} (\sO_{\sF}^{op},\sM)$, $G\sM$ is also Quillen equivalent to
the presheaf category $\mathbf{Pre} (\sO_{\sF}\otimes\sD, \sV)$. 
\end{thm}

We have a canonical $\sV$-functor
$\tau\colon \sO_{\sF}\otimes\sD\rtarr G\sM$ that sends  the pair
$(G/H,d)$ to $G/H\odot \de d$.  The maps of 
enriched hom objects are given by the tensor bifunctor
\[ \odot\colon \ul{G\sV}(G/H,G/K)\otimes {\sD}(d,e) 
\rtarr
\ul{G\sM}(G/H\odot \de d, G/K\odot \de e). \] 
Let $\sF\sD$ denote the full $\sV$-subcategory of $G\sM$ whose
objects are the $G/H\odot \de d$ with $H\in \sF$.  Since $\ta$
lands in $\sF\sD$, it specifies a $\sV$-functor
\[ \ta\colon \sO_{\sF}\otimes\sD \rtarr \sF\sD.\]
Even when $\de$ is the inclusion of a full subcategory, it is unclear to us whether or not $\ta$ is a weak equivalence 
in the sense defined in \cite[Definition~2.3(iii)]{GM}. 

In any case, this is an important example where the domain of the presheaf category that arises most naturally 
in answering \cite[Question 0.2 or 0.4]{GM} is not a full $\sV$-subcategory. 
We have Quillen adjunctions of $\sF$-model categories
\[ \xymatrix@1{
\mathbf{Pre} (\sO_{\sF}\otimes\sD,\sV) \ar@<0.4ex>[r]^-{\ta_*} & \mathbf{Pre} (\sF\sD,\sV) \ar@<0.4ex>[l]^-{\ta^*}}  
\ \ \
\text{and} \ \ \ 
\xymatrix@1{\mathbf{Pre} (\sF\sD,\sV) \ar@<0.4ex>[r]^-{\bT} & G\sM \ar@<0.4ex>[l]^-{\bU}.}\\ \]
A check of definitions using (\ref{keyGthree}) shows that the 
composite Quillen adjunction is the Quillen equivalence of \myref{neat}, 
but we do not know whether or not these Quillen adjunctions themselves 
can also be expected to be Quillen equivalences.

\section{Modules over cocommutative DG Hopf algebras}\label{Hopfalg}

\subsection{The general context}  Let $R$ be a commutative ring.  We specialize the general theory to the
category $\sV = R\text{-}\bf{Mod}$ of $\bZ$-graded chain complexes over $R$. 
It is a cosmos with product $\otimes$, unit $R$, and internal hom  $\Hom_R$. 
Differentials lower degree; replacing $X_n$ by $X^{-n}$ would reverse this convention. 
To avoid distraction, the reader may prefer to restrict $R$ to be a field.

In this setting, a Hopf group $A$ in $\sV$ is a cocommutative differential graded 
$R$-Hopf algebra.\footnote{We change notation from $G$ to $A$
for psychological rather than mathematical reasons.}  From the point of view
of \S2 and \S3, we take $\sV = \sM$ to be $R\text{-}\bf{Mod}$.  Then $G\sV$ 
becomes the category $A\text{-}\bf{Mod}$ of left DG $A$-modules.  
Note that we have successively simpler cases where we take
the differential on $A$ to be trivial and when we take $A$ to be concentrated 
in degree $0$.  We can specialize by taking $A$ to be the group  ring of a 
discrete group, but our context is vastly more general.  

We give $R\text{-}\bf{Mod} $ the model structure whose weak equivalences, fibrations, 
and cofibrations are the quasi-isomorphisms, the degreewise epimorphisms, 
and the degreewise split monomorphisms with cofibrant cokernel.  Cofibrant 
objects are degreewise projective, and the converse holds for bounded below 
objects.  This model structure is compactly generated.  Canonical generating
sets $\mathcal{I}_R$ and $\mathcal{J}_R$ are given by the inclusions 
$S^{n-1}_R\rtarr D^n_R$ and $0\rtarr D^n_R$ for $n\in \bZ$.  Here $S^n_R$ is
$R$-free on one generator of degree $n$, with zero differential, and $D^n_R$ 
is $R$-free on generators of degrees $n$ and $n-1$, with $d_n = \id$. See
for example \cite[\S18.4]{MP} for details and alternative model structures.  

Ignoring the Hopf algebra structure and thus generalizing to DGAs over $R$,
we can give $A\text{-}\bf{Mod}$ the model structure whose weak equivalences and fibrations 
are the maps which are weak equivalences and fibrations when regarded as
maps in $R\text{-}\bf{Mod} $.  That is, we take the model structure induced by the 
underlying $R$-module functor $\bR \colon A\text{-}\textbf{Mod}\rtarr R\text{-}\textbf{Mod} $. Then $(\bF , \bR)$ is a 
Quillen adjunction, where $\bF\colon R\text{-}\textbf{Mod} \rtarr R\text{-}\textbf{Mod}$ is the extension
of scalars functor that sends $X$ to $A\otimes_R X$.  This model structure
is also compactly generated. Generating sets $\mathcal{I}_A$ and 
$\mathcal{J}_A$ are obtained by applying $\bF$
to the maps in $\mathcal{I}_R$ and $\mathcal{J}_R$. Other model structures 
defined in \cite{BMR} could also be used.  From the point of view of this paper,
we are here taking $\sF=\{e\}$ and thus describing naive or Borel equivariant homotopy theory,
for which the Hopf algebra structure is irrelevant.

Now return to our Hopf group $A$.  As in \S2 and \S3, we have the categories
	\begin{itemize}
		\item{}$A\text{-}\bf{Mod}$ of $A$-modules and $A$-maps
		\item{}$(R\text{-}\bf{Mod})_A$ of $A$-modules and $R$-maps, with $A$ acting by conjugation
	\end{itemize}
Since the $\Hom_A(M,N)$ are chain complexes of $R$-modules, $A\text{-}\bf{Mod}$ 
is enriched over $R\text{-}\bf{Mod}$,  and  $(R\text{-}\bf{Mod})_A$ is enriched over $A\text{-}\bf{Mod}$.

Let $\sF$ be a set, not necessarily a family, of sub-Hopf algebras $B$ of $A$ that contains $R = S^0_R$. Let $\sO_{\sF}$ be the full 
subcategory of $A$-$\bf{Mod} $ whose objects are the ``orbits'' $A/\!/B = A \otimes_B R$ for $B \in \sF$.  
The $A$-module
$A/\!/B$ is $A/A\cdot IB$ where $IB = \Ker(\epz\colon B\rtarr R)$ is the augmentation ideal of $B$.   Thus we set $ab = 0$ if $\text{deg} (b) \neq 0$
and $ab = a \epz (b)$ if $\text{deg} (b) = 0$. 

\begin{rem}\mylabel{RGood} If $R$ is a field and $A$ is concentrated in nonnegative degrees, then $(A/\!/B)^C$ is also concentrated in nonnegative degrees for any sub-Hopf algebras 
$B,C \subset A$. Since we are working over a field, they are all cofibrant in $R\text{-}\mathbf{Mod}$.  Therefore every $\sF$ is $A\text{-}\mathbf{Mod}$-good.
\end{rem}

The general theory specializes to 
give results analogous to those in the  topological context of \S\ref{GTop}, but we now need the more general context that we have
developed to deal with the $R$-$\bf{Mod} $-category $A$-$\bf{Mod}$. We shall prove a version of theorem \ref{GObjAsPrshvs} for 
categories of modules over general $A$.

\subsection{Colimits and passage to fixed points}

In order to set up model structures and to establish the Quillen equivalence between $\mathbf{Pre}(\sO_\sF,R\text{-}\mathbf{Mod})$ and $A\text{-}\mathbf{Mod}$, we must show that taking fixed points commutes with enough colimits. We first describe the ``fixed points'' of an $A$-module $M$. For any sub Hopf-algebra $B \subset A$, we have
\[ M^B = \{ m \, | \, (b-\epz b)m = 0 \ \ \text{for all} \ \ b\in B\}. \]
Thus  $(M^B)_n$ is the intersection of the kernels of the maps $M_{n} \rtarr M_{n+p}$ given by multiplication by $b-\epz b$ for
$b\in B_{p}$.   Of course, $\epz b = 0$ unless $b\in B_0$. 

\begin{lem} Let $A$ be a cocommutative DG Hopf algebra over $R$. For any sub-Hopf algebra $B \subset A$, the fixed point functor $(-)^B : A\text{-}\mathbf{Mod} \rtarr R\text{-}\mathbf{Mod}$ preserves arbitrary coproducts, and sequential colimits of monomorphisms (indexed over any ordinal).
\end{lem}

The proof is straightforward.  Note that coproducts here are direct sums, whereas sequential colimits are constructed by taking colimits of underlying sets. 

For $M\in A\text{-}\mathbf{Mod}$ and $N\in R\text{-}\mathbf{Mod}$, we have 
$$(M\otimes \epz^*N)^B \iso M^B\otimes N$$ 
when $N$ is a degreewise free $R$-module, such as the domain or codomain of a generating cofibration or acyclic cofibration.  In particular, we have the following observation.

\begin{lem}Let $A$ be a cocommutative DG Hopf algebra over $R$. For any sub-Hopf algebras $B, C \subset A$, and for any integer $n \in \bZ$, we have isomorphisms
	\[
	(A/\!/B \otimes S^n_R)^C \cong (A/\!/B)^C \otimes S^n_R \quad\text{and}\quad (A/\!/B \otimes D^n_R)^C \cong (A/\!/B)^C \otimes D^n_R.
	\]
\end{lem}

We now consider pushouts. In general, commuting $(-)^B$ with pushouts in $A\text{-}\mathbf{Mod}$ is difficult, but we need only consider pushouts where one leg is a generating cofibration $A/\!/B \otimes S^{m-1}_R \rtarr A/\!/B \otimes D^m_R$ since relative cell complexes can be constructed one cell at a time.

\begin{lem} Let $A$ be a cocommutative DG Hopf algebra over $R$.  If the diagram
	\begin{center}
		\begin{tikzpicture}[node distance=2cm]
			\node(S) {$A/\!/B \otimes S^{m-1}_R$}; 
			\node(D) [below of=S] {$A/\!/B \otimes D^m_R$}; 
			\node(X) [right of=S] {$X$};
			\node(Y) [below of=X] {$Y$};
			\path[->]
			(S) edge [left] node {$i$} (D)
			(S) edge [above] node {$f$} (X)
			(D) edge [below] node {$g$} (Y)
			(X) edge [right] node {$j$} (Y)
			;
		\end{tikzpicture}
	\end{center}
is a pushout in $A\text{-}\mathbf{Mod}$, then applying the functor $(-)^C$ to it gives a pushout in $R\text{-}\mathbf{Mod}$  for any sub-Hopf algebra $C$ of $A$. 
\end{lem}

\begin{proof}
Fix an integer $n \in \bZ$ and consider what happens in degree $n$. Writing $R_k$ for a copy of $R$ in degree $k$, we see that this pushout square can be identified with
	\begin{center}
		\begin{tikzpicture}[node distance=2cm]
			\node(S) {$(A/\!/B)_{n-m+1} \otimes R_{m-1}$};
			\node(D) [below of=S] {$\Big( (A/\!/B)_{n-m+1} \otimes R_{m-1} \Big) \oplus \Big( (A/\!/B)_{n-m} \otimes R_m \Big)$};
			\node(X) at (6.8,0) {$X_n$};
			\node(Y) [below of=X] {$X_n \oplus \Big( (A/\!/B)_{n-m} \otimes R_m \Big)$};
			\path[->]
			(S) edge [left] node {} (D)
			(S) edge [above] node {$f_n$} (X)
			(D) edge [below] node {$f_n \oplus \text{id}$} (Y)
			(X) edge [right] node {} (Y)
			;
		\end{tikzpicture}
	\end{center}
where the vertical maps are the inclusions of summands. Moreover, the $A$-action respects both splittings, since the original pushout was taken in $A\text{-}\mathbf{Mod}$. Thus, this pushout is preserved when we pass to the square of submodules on which $C$ acts through the augmentation, and therefore the ``fixed point'' functor $(-)^C$ preserves the pushout of $A$-modules.
\end{proof}

\begin{rem}We have an analogous result 
when the left leg is the (unique) map $A/\!/B \otimes 0 \rtarr A/\!/B \otimes D^m_R$, because in such a case, $Y \cong (A/\!/B \otimes D^m_R) \oplus X$ and we already know that $(-)^C$ preserves direct sums.
\end{rem}

\subsection{Model categorical results}\label{subsubsec:GHom}

We prove theorem \ref{GObjAsPrshvs} for $A$-modules. Let  $\bT : \mathbf{Pre}(\sO_\sF,R\text{-}\mathbf{Mod}) \rightleftarrows A\text{-}\mathbf{Mod} : \bU$
be the usual adjunction, and recall the definitions of the level $\sF$-model structure on $\mathbf{Pre}(\sO_\sF,R\text{-}\mathbf{Mod})$ and the $\sF$-model structure on $A\text{-}\mathbf{Mod}$ given in section \ref{GMModel}.

\begin{thm}\label{thm:GHompresh} The level $\sF$-equivalences, level $\sF$-fibrations, and the resulting 
cofibrations give $\mathbf{Pre}(\sO_{\sF}, R\text{-}\mathbf{Mod})$ a compactly generated, right proper $R\text{-}\bf{Mod}$-model
structure with generating cofibrations $F_{\sF}\mathcal{I}_R$ 
and generating acyclic cofibrations $F_{\sF}\mathcal{J}_R$. If $\sF$ is a good set of sub-Hopf algebras, then 
$\mathbf{Pre}(\sO_\sF,R\text{-}\mathbf{Mod})$ is also left proper and every cofibration in 
$\mathbf{Pre}(\sO_\sF,R\text{-}\mathbf{Mod})$ is a levelwise cofibration.
\end{thm}

\begin{proof} To show the model structure exists, use the adjunction $(\bU(A/\!/B) \otimes (-) , \text{ev}_{A/\!/B})$ and the fact that $\text{ev}_{A/\!/B}$ preserves colimits to reduce to the smallness of $S^{m}_R$ and $0$ with respect to relative $\{ (A/\!/B)^C \otimes i \}$ and $\{ (A/\!/B)^C \otimes j\}$-cell complexes, and the acyclicity condition for relative $\{ (A/\!/B)^C \otimes j\}$-cell complexes in $R\text{-}\mathbf{Mod}$.
\end{proof}

\begin{thm}The $\sF$-equivalences, $\sF$-fibrations, and the resulting cofibrations
give $A$-$\bf{Mod}$ a compactly generated, right proper $R$-$\bf{Mod} $-model category structure with generating cofibrations $\sF\mathcal{I}_R$ and generating acyclic cofibrations $\sF\mathcal{J}_R$. If $\sF$ is a good set of sub-Hopf algebras, then $A\text{-}\mathbf{Mod}$ is also left proper, and the functors $(-)^B$ preserve cofibrations. 
\end{thm}

\begin{proof}The argument is similar to the proof of theorem \ref{thm:GHompresh}. One uses the adjunction $(A/\!/B \otimes (-) , (-)^B)$ and the non-formal fact that $(-)^B$ preserves sequential colimits of monomorphisms to reduce the smallness of $\sF\mathcal{I}_R$ and $\sF\mathcal{J}_R$ to the smallness of $S^m_R$ and $0$. Then, since $\bU$ takes relative $\sF\mathcal{J}_R$-cell complexes to relative $F_\sF \mathcal{J}_R$-cell complexes, the acyclicity condition for $\sF\mathcal{J}_R$ is inherited from $F_\sF \mathcal{J}_R$.
\end{proof}

\begin{rem}  For left properness, recall from \myref{RGood} that the goodness hypothesis in the above results holds
quite generally.
\end{rem}

\begin{rem}Since $A\text{-}\mathbf{Mod}$ is a cosmos, one may ask whether it is a monoidal model category. We do not believe this is true in general, but it is true when $A$ is commutative and we use $\sF = \{R\}$ \cite[Theorem 3.3]{BMR}. 
\end{rem}

\begin{thm} \label{thm:DGHopfpresh} The functors
	\[
	\xymatrix@1{\mathbf{Pre} ( \sO_{\sF},R\text{-}\bf{Mod})\ar@<.4ex>[r]^-{\bT} & A\text{-}\bf{Mod} \ar@<.4ex>[l]^-{\bU}\\}
	\]
give an $R\text{-}\mathbf{Mod}$-enriched Quillen equivalence.
\end{thm}

\begin{proof} The usual argument for spaces works verbatim.
\end{proof}

\begin{rem} We can generalize the theory of this section by demanding an action of a discrete group $G$ on our Hopf
group $A$ through automorphisms of algebras and using twisted modules (e.g. \cite{GM3, GMM}).  A quite different
generalization would replace equivariant DG algebras by equivariant DG categories.  This section can be viewed as a 
modest contribution to the nascent field of equivariant homological algebra.
\end{rem}

\section{Equivariant simplicial model categories}\label{secsset}

Since simplicial enrichment is the one most commonly used, we would be remiss 
not to show how our theory applies to equivariant simplicial model categories.
Here we take $\sV$ to be the closed cartesian monoidal category $\mathbf{sSet}$ of simplicial sets.
We give $\mathbf{sSet}$ its usual Quillen model structure and write 
$\mathcal{I} = \{ \partial\Delta^n \rtarr \Delta^n \}$ and $\mathcal{J} = \{ \Lambda^n_k \rtarr \Delta^n\}$ for its sets of generating cofibrations and acyclic cofibrations; other choices 
are possible. We take $G$ to be a  simplicial group. Less generally, we can take a group $G$ and regard it as a 
discrete simplicial group, that is, a discrete group regarded as a constant simplicial set; according to our general theory, it should be denoted $I[G]$.

Let $\sF$ be a set, not necessarily a family, of simplicial subgroups $H$ of $G$ that contains $e$. Let $\sO_\sF$ be the full simplicial subcategory of 
$G\mathbf{sSet}$ whose objects are the orbits $G/H = G \times_H *$ for $H \in \sF$. Since every simplicial set is cofibrant, every such set of subgroups
is automatically $\bf{sSet}$-good.   Since the colimits appearing in the definition of $\sF$-cell complexes are colimits of inclusions, $\sF$ is very 
$\bf{sSet}$-good when passage to fixed points preserves inclusions.  We shall see that this always holds when $G$ is discrete.  As usual, the most 
interesting examples 
of $\sF$ are $\sA \ell \ell$ and $\{e\}$. 

For now, we take $\sM = \sV = \mathbf{sSet}$, and consider the actions of $G$ on simplicial sets $T$. We shall consider more general $\sM$ in \S\ref{Mouch}.  We have the $\mathbf{sSet}$-category $G\mathbf{sSet}$ of $G$-actions and $G$-maps and the $G\mathbf{sSet}$-category $\mathbf{sSet}_G$ of $G$-actions and all simplicial maps, with $G$ acting by conjugation on hom objects. Then 
$$\ul{\mathbf{sSet}}_G(T,T')^G \cong \ul{G\mathbf{sSet}}(T,T').$$

\subsection{Colimits and passage to fixed points}

As usual, we consider the interaction between colimits and fixed-point functors. We begin with a description of the simplices of $T^H$ for a $G$-action $T$ and a simplicial subgroup $H \subset G$.
	\[
	(T^H)_n = \left\{ x \in T_n \, \Big| \,   \phi^*(x) \in (T_q)^{H_q} 
	\text{ for all } q \geq 0 \text{ and } \phi \colon  [q] \rtarr [n]  \right\} 
	\]
Specializing to $q=n$ and $\phi = \text{id} : [n] \rtarr [n]$, we see that $(T^H)_n \subset (T_n)^{H_n}$, but equality need not hold in general. Indeed, if $x \in (T_n)^{H_n}$, then for every $\phi : [q] \rtarr [n]$, we are guaranteed that $\phi^*(x) \in (T_q)^{\phi^*(H_n)}$, but elements of $H_q \setminus \phi^*(H_n)$ need not fix $\phi^*(x)$. However, the equality $(T^H)_n = (T_n)^{H_n}$ does hold if all $\phi^*$ for $G$ are surjective. In particular, this is true if $G$ is a discrete simplicial group, and in this case 
$$\ul{G\mathbf{sSet}}(G/K,G/H) = [G/H]^K,$$
regarded as a constant simplicial set.

The following lemma is straightforward.

\begin{lem} Let $G$ be an arbitrary simplicial group. For any simplicial subgroup $H \subset G$, the fixed point functor $(-)^H : G\mathbf{sSet} \rtarr \mathbf{sSet}$ preserves coproducts and sequential colimits of monomorphisms (indexed over any ordinal). Moreover, for any simplicial set $X$ regarded as a $G$-trivial $G$-simplicial set, 
$$(T \times X)^H \cong T^H \times X.$$
\end{lem}

We also have the following analogue to Lewis' observation for $G$-spaces.

\begin{lem} Let $G$ be a discrete simplicial group. If 
	\begin{center}
		\begin{tikzpicture}[node distance=2cm]
			\node(a) {$A$};
			\node(b) [below of=a] {$B$};
			\node(x) [right of=a] {$X$};
			\node(y) [below of=x] {$Y$};
			\path[->]
			(a) edge [above] node {$f$} (x)
			(x) edge [right] node {$j$} (y)
			(b) edge [below] node {$g$} (y)
			;
			\path[>->]
			(a) edge [left] node {$i$} (b)
			;
		\end{tikzpicture}
	\end{center}
is a pushout in $G\mathbf{sSet}$ in which $i$ is a monomorphism, then the square obtained by applying $(-)^H$ is a 
pushout in $\mathbf{sSet}$ for any subgroup $H$ of $G$.
\end{lem}
\begin{proof} In every dimension $n \geq 0$, we have a splitting $Y_n \cong (B_n \setminus A_n) \sqcup X_n$. The action of $H_n = H$ respects this splitting, but the simplicial structure on $Y$ does not. However, since $H$ is discrete, it follows that $(T^H)_n = T_n^{H_n} = T_n^H$, and hence
	\[
	(Y^H)_n \cong Y_n^{H} \cong (B_n^{H} \setminus A_n^{H}) \sqcup X_n^{H} = \Big( (B^H)_n \setminus (A^H)_n \Big) \sqcup (X^H)_n. \qedhere
	\]
\end{proof}

\begin{ouch0}\mylabel{Nono} We  indicate what can go wrong in the preceding argument when $G$ is not discrete, using a specific example. Consider the problem of attaching an equivariant $1$-cell to a point,
	\begin{center}
		\begin{tikzpicture}[node distance=2cm]
			\node(a) {$G/H \times \partial \Delta^1$};
			\node(b) [below of=a] {$G/H \times \Delta^1$};
			\node(x) [right of=a] {$*$};
			\node(y) [below of=x] {$Y$};
			\path[->]
			(a) edge [above] node {} (x)
			(x) edge [right] node {} (y)
			(b) edge [below] node {} (y)
			;
			\path[>->]
			(a) edge [left] node {$i$} (b)
			;
		\end{tikzpicture}
	\end{center}
and look at the $1$-simplices of $Y$. The splitting of $Y_1$ is $Y_1 \cong ({G_1/H_1} \times \{ \text{id} \} ) \sqcup *$, where $\text{id} : [1] \rtarr [1]$ in the simplex category. Now let $K$ be a simplicial subgroup of $G$, and take $Y^K$. Then $(Y^K)_1$ consists of the point $*$ and those pairs $(gH_1 , \text{id})$ such that $\phi^*(gH_1) \in (G_{k+1}/H_{k+1})^{K_{k+1}}$ for every iterated degeneracy map 
$$\phi = s_{i_1} \circ \cdots \circ s_{i_k} \colon  [k+1] \rtarr [1].$$ Indeed, any simplicial operator $\phi$ with target
$[1]$ can be written as a composite $\phi = d_{j_1} \circ \cdots \circ d_{j_l} \circ s_{i_1} \circ \cdots \circ s_{i_k}$. If any face map $d_j$ appears in this decomposition, then $\phi^*(gH_1,\text{id}) = * $ and hence $\phi^*(gH_1,\text{id})$ is automatically in $Y_q^{K_q}$. If no $d_j$ appears, then $\phi^*(gH_1,\text{id})$ is in $Y_q^{K_q}$ if and only if  
$\phi^*(gH_1)$ is in $(G_q/H_q)^{K_q}$ since $\phi$ preserves the complement of the image of the inclusion $i : G/H \times \partial \Delta^1 \rtarr G/H \times \Delta^1$.

On the other hand, the set $(G/H \times \Delta^1)^K_1 \setminus (G/H \times \partial\Delta^1)^K_1$ consists of those pairs $(gH_1 , \text{id})$ such that for \emph{any}  $\phi : [q] \rtarr [1]$, we have $\phi^*(gH_1) \in (G_q/H_q)^{K_q}$. It follows that we have an inclusion $\Big( (G/H \times \Delta^1)^K_1 \setminus (G/H \times \partial\Delta^1)^K_1 \Big) \coprod  *^K_1 \subset (Y^K)_1$, and by making suitable choices of $G$, $H$, and $K$, we can make it a proper inclusion.
\end{ouch0}

\subsection{Model categorical results}

We now prove an analogue of  \myref{GObjAsPrshvs} for a discrete simplicial group $G$. Let 
$$\bT \colon \mathbf{Fun}(\sO^{op}_\sF,\ul{\mathbf{sSet}}) \rightleftarrows G\mathbf{sSet} \colon \bU$$ 
be the usual adjunction, and recall the definitions of the level $\sF$-model structure on $\mathbf{Fun}(\sO^{op}_\sF,\ul{\mathbf{sSet}})$ and the $\sF$-model structure on $G\mathbf{sSet}$ given in \S\ref{GMModel}. The proofs of the following statements are essentially identical to those given in \S\ref{subsubsec:GHom}.

\begin{thm}Suppose $G$ is a discrete simplicial group. The level $\sF$-equivalences, level $\sF$-fibrations, and the resulting cofibrations give $\mathbf{Fun}(\sO^{op}_\sF , \ul{\mathbf{sSet}})$ a compactly generated, proper simplicial model structure with generating cofibrations $F_\sF \mathcal{I}$ and generating acyclic cofibrations $F_\sF \mathcal{J}$. Every cofibration of $\mathbf{Fun}(\sO^{op}_\sF,\ul{\mathbf{sSet}})$ is a levelwise cofibration.
\end{thm}

Observe that every cofibration in $G\mathbf{sSet}$ is a monomorphism, and hence $\bU$ preserves pushouts, provided one of its legs is a cofibration. Thus, we have the following result.

\begin{thm} Suppose $G$ is a discrete simplicial group. The $\sF$-equivalences, $\sF$-fibrations, and the resulting cofibrations give $G\mathbf{sSet}$ a compactly generated, proper simplicial model structure, with generating cofibrations $\sF\mathcal{I}$ and generating acyclic cofibrations $\sF\mathcal{J}$. Moreover, the functors $(-)^H$ preserve cofibrations.
\end{thm}

We do not believe that $G\mathbf{sSet}$ is a monoidal model category in general.

\begin{thm}Suppose that $G$ is a discrete simplicial group. Then the functors
	\[
	\xymatrix@1{\mathbf{Fun} ( \sO^{op}_{\sF},\ul{\mathbf{sSet}})\ar@<.4ex>[r]^-{\bT} &G {\mathbf{sSet}} \ar@<.4ex>[l]^-{\bU}\\}
	\]
are a simplicial Quillen equivalence.
\end{thm}

\begin{rem}  For a general simplicial group $G$ one might try to prove the preceding theorem using a bar construction, as in Elmendorf's argument  \cite{E}. In such a proof, one would have to prove that taking $H$-fixed points commutes with geometric realization, and this again runs into problems. Indeed, $(-)^H$ is an \emph{enriched} right Kan extension, and while the geometric realization functor for bisimplicial sets is isomorphic to pullback along the diagonal $\Delta \rtarr \Delta \times \Delta$ as an \emph{unenriched} functor, this is not true when we enrich.
\end{rem}

\subsection{Actions in more general simplicial model categories $\sM$}\label{Mouch}

We now partially generalize some of our results for simplicial sets to categories $\sM$ enriched over simplicial sets. We require some rather annoying assumptions to do this.   The reason is that orbits and fixed points are 
colimits and limits.  Commuting them in special cases is central to our philosophy, and things that work trivially for simplicial sets cannot be expected to work at all in general categories enriched in simplicial sets.

\begin{asses}\mylabel{assump:simpl} We assume the following conditions.
	\begin{enumerate}[(i)]
		\item{}$G$ is a finite group, regarded as a discrete simplicial group.
		\item{}$\sM$ is a locally finitely presentable, cofibrantly generated, simplicial model category with generating cofibrations $\mathcal{I}$ and generating acyclic cofibrations $\mathcal{J}$.
		\item{}the Lewis condition holds: for any subgroup $H \in \sF$, the fixed point functor $(-)^H : G\sM \rtarr \sM$ preserves pushouts one leg of which is in $\sF \mathcal{I}$ or $\sF \mathcal{J}$.
		\item{}The following two conditions hold:
			\begin{itemize}
				\item[(a)] finite coproducts in $\sM$ are ``disjoint'': for any finite set $I$ and objects $(M_i)_{i \in I}$ in $\sM$, each inclusion $M_i \rtarr \coprod_{i \in I} M_i$ is a monomorphism, and if $i \neq j$, then the pullback (``intersection'')  of $M_i \rtarr \coprod_{i \in I} M_i$ and $M_j \rtarr \coprod_{i \in I} M_i$ is the initial object in $\sM$, and
				\item[(b)]  equalizers split over finite coproducts in $\sM$: for any finite set $I$, objects $(A_i)_{i \in I}$ and $B$ in $\sM$, and morphisms $f_i : A_i \rtarr B$ and $g_i : A_i \rtarr B$, the dashed map in the diagram below
					\begin{center}
		\begin{tikzpicture}
			\node(A) {$\coprod_{i \in I} \text{eq}(f_i,g_i)$};
			\node(B) at (6,0) {$\text{eq}([f_i]_{i \in I} , [g_i]_{i \in I})$};
			\node(C) at (3,-1.5) {$\coprod_{i \in I} A_i$};
			\node(D) at (3,-3) {$B$};
			\path[dashed, ->]
			(A) edge [above] node {$\cong$} (B)
			;
			\path[->]
			(A) edge [below left] node {} (C)
			(B) edge [below right] node {} (C)
			(C.253) edge [left] node {$[f_i]_{i \in I}$} (D.110)
			(C.286) edge [right] node {$[g_i]_{i \in I}$} (D.70)
			;
		\end{tikzpicture}
					\end{center}
				is an isomorphism.
			\end{itemize}
	\end{enumerate}
\end{asses}

\begin{rem}Condition (iv.a)  holds in the most interesting cases, but a counterexample can be obtained by considering the coproduct of $\bF_2$ and $\bF_3$ in the category of commutative rings. On the other hand, we only expect condition (iv.b) to hold in sufficiently ``space-like'' categories. It fails for modules over a ring $R$: consider the equalizer of the identity and twist maps $\t{id}, \gamma : A \oplus A \xymatrix@1 { \ar@<.4ex>[r] \ar@<-.4ex>[r] & } A \oplus A$.
\end{rem}

Under these assumptions and our standard model theoretic assumptions,  we shall prove a version of theorem \ref{GObjAsPrshvs}.

\subsubsection{Colimit preservation lemmas}

The conditions in \myref{assump:simpl} ensure that certain hom functors $\t{hom}(A,-)$ preserve filtered colimits, but we must relate our assumptions to the colimit preservation properties of $(-)^H$. To start, observe that for any 
$M \in G\sM$ and subgroup $H \subset G$:
	\[
	M^H = \text{eq} \Bigg( M \xymatrix@1 { \ar@<.4ex>[r]^-{\Delta} \ar@<-.4ex>[r]_-{(h)}& } \prod_{h \in H} M \Bigg),
	\]
where $(h)$ is  the composite of $\DE$ and the morphism given by multiplication by $h$ on the $h$th copy of  $M$.  This description makes  essential use of the fact that $G$ is discrete.

\begin{lem} The functor $(-)^H : G\sM \rtarr \sM$ preserves directed colimits.
\end{lem}

\begin{proof} The fixed point functor is a finite limit since $H$ is finite, so this holds since finite limits and filtered colimits commute in any locally finitely presentable category.
\end{proof}

Next, we clarify the role of (iv.a) and (iv.b) of \myref{assump:simpl}. Note that since the orbit $G/H$ is also discrete, we have
	\[
	G/H \odot M \cong \coprod_{gH \in G/H} M
	\]
for $M \in \sM$. The $G$-action is obtained by permuting the copies of $M$, and the point is that if coproducts are disjoint and taking fixed points splits over them, then the $K$-fixed summands of $G/H \odot M$ correspond to $K$-fixed elements of $G/H$.

\begin{lem}  Conditions (iv.a) and (iv.b) imply that $(G/H \odot M)^K \cong (G/H)^K \odot M$
for any subgroups $H,K \subset G$ and any object $M\in \sM$. 
\end{lem}

\begin{proof} By (iv.b), we may compute the equalizer $(\coprod_{G/H} N)^K$ one copy of $M$ at a time. 
Write $\iota_{gH} : M \rtarr \coprod_{G/H} M$ for the inclusion, and consider the $gH$-summand. We must compute the equalizer
	\[
	\text{eq} \,\,\,\xymatrix@1 { \ar@{>->}[r]^-{e} & } M \xymatrix@1 { \ar@<.4ex>[r]^-{\langle \iota_{gH} \rangle_k} \ar@<-.4ex>[r]_{\langle \iota_{kgH} \rangle_k} & } \prod_{k \in K} \Big(\coprod_{G/H} M \Big).
	\]
If $gH \in (G/H)^K$, then the two parallel morphisms are equal, and then $\text{eq} \cong M$. If $gH \notin (G/H)^K$, then we may choose $k \in K$ such that $kgH \neq gH$, and after projecting along $\pi_k \colon \prod_{k \in K} \Big(\coprod_{G/H} M \Big) \rtarr \coprod_{G/H} M$, we see that $e : \text{eq} \rtarr M$ must equalize the pair $\iota_{gH} , \iota_{kgH} \colon M \xymatrix@1 { \ar@<.4ex>[r] \ar@<-.4ex>[r] & } \coprod_{G/H} M$. Thus $\text{eq}$ is a subobject of $\iota_{gH} \cap \iota_{kgH} = \varnothing$, hence $\text{eq} \cong \varnothing$.
\end{proof}

\begin{rem}\mylabel{rem:cellconv} In what follows, let us agree to only use cell complexes for which
	\begin{itemize}
		\item{}a single cell is attached at each successor stage, and
		\item{}the final (transfinite) sequential colimit is taken over an infinite ordinal.
	\end{itemize}
With this convention, we see that  \myref{assump:simpl} imply that the fixed point functors $(-)^H : G\sM \to \sM$ preserve all tensors and colimits appearing in the construction of relative $\sF \mathcal{I}$ and relative $\sF \mathcal{J}$-cell complexes.
\end{rem}

\begin{warn} If we had only assumed that $\sM$ is locally presentable, say for some regular cardinal $\lambda > |G|$, then the same line of argument would show that $(-)^H$ preserves $\lambda$-filtered colimits. However, this does not imply that
$(-)^H$  preserves all colimits used to construct relative cell complexes, because when we form transfinite composites, we must take a sequential colimit at every limiting stage.
\end{warn}

\subsubsection{Model categorical results}

Fix a set $\sF$ of subgroups of $G$ that contains $e$. As promised, we prove a version of \myref{GObjAsPrshvs}. Let $(\bT,\bU)$ be the usual adjunction, and recall the definitions of the level $\sF$-model structure on $\mathbf{Fun}(\sO^{op}_\sF,\sM)$ and the $\sF$-model structure on $G\sM$ given in \S\ref{GMModel}. As above, we assume \myref{assump:simpl} and we only use cell complexes as described in \myref{rem:cellconv}.

\begin{thm}The level $\sF$-equivalences, level $\sF$-fibrations, and the resulting cofibrations give $\mathbf{Fun}(\sO^{op}_\sF,\sM)$ a locally finitely presentable, cofibrantly generated, simplicial model structure with generating cofibrations $F_\sF \mathcal{I}$ and generating acyclic cofibrations $F_\sF \mathcal{J}$. Every cofibration of $\mathbf{Fun}(\sO^{op}_\sF,\sM)$ is a levelwise cofibration, and if $\sM$ is left or right proper, then so is $\mathbf{Fun}(\sO^{op}_\sF,\sM)$.
\end{thm}

\begin{proof}  The local finite presentability of $\sM$ implies that the sets $F_\sF \mathcal{I}$ and $F_\sF \mathcal{J}$ admit the small object argument, and since every $(G/H)^K$ is cofibrant in $\mathbf{sSet}$, it follows that $F_{G/H} \odot j$ is a levelwise acyclic cofibration for every $j \in \mathcal{J}$ and $H \in \sF$. Thus, $F_\sF \mathcal{J}$ satisfies the acyclicity condition, and the level $\sF$-model structure exists. To see that the presheaf category inherits local finite presentability, note that the underlying category $\mathbf{Fun}_{\mathbf{sSet}}(\sO^{op}_\sF,\sM)_0$ of $\mathbf{Fun}_{\mathbf{sSet}}(\sO^{op}_\sF,\sM)$ is isomorphic to the ordinary functor category $\mathbf{Fun}_{\mathbf{Set}}( (\sO^{op}_\sF)_0 , \sM_0)$, because $\sO^{op}_\sF$ is discrete. Finally, every set $\sF$ of subgroups is good under simplicial enrichment, hence $\mathbf{Fun}(\sO^{op}_\sF , \sM)$ inherits left and right properness from $\sM$.
\end{proof}

\begin{thm}The $\sF$-equivalences, $\sF$-fibrations, and the resulting cofibrations give $G\sM$ a locally finitely presentable, cofibrantly generated, simplicial model structure with generating cofibrations $\sF \mathcal{I}$ and generating acyclic cofibrations $\sF \mathcal{J}$. Moreover, for every $H \in \sF$, the functor $(-)^H$ preserves cofibrations, and if $\sM$ is left or right proper, then so is $G\sM$.
\end{thm}

\begin{proof} Local finite presentability of $\sM$ and the fact that $(-)^H$ preserves directed colimits imply that $\sF \mathcal{I}$ and $\sF \mathcal{J}$ admit the small object argument. \myref{assump:simpl} guarantee that $\bU$ takes relative $\sF \mathcal{J}$-cell complexes to relative $F_\sF \mathcal{J}$-cell complexes, and hence $\sF\mathcal{J}$ inherits the acyclicity condition from $F_\sF \mathcal{J}$. Thus the $\sF$-model structure exists on $G\sM$. The category 
$G\sM$ inherits local finite presentability and properness from $\sM$ as above.
\end{proof}

\begin{thm} The functors
	\[
	\xymatrix@1 { \mathbf{Fun}(\sO^{op}_\sF , \sM) \ar@<.4ex>[r]^-{\bT} & G\sM \ar@<.4ex>[l]^-{\bU} }
	\]
are a simplicial Quillen equivalence.
\end{thm}

\begin{proof} Given our conventions on cell complexes, the usual argument works.
\end{proof}

\section{Appendix: Categorical explanations and amplifications}\label{Catcat}

\subsection{Closed subgroups}\label{closed}

In homotopy theory, it is standard to restrict attention to closed subgroups.  We give that some categorical perspective and suggest a definition of closed subgroups of Hopf groups.

Recall that a monomorphism $m : X \rtarr Y$ in any category is \emph{regular} if it is an equalizer of some pair of arrows $Y \xymatrix@1 { \ar@<.4ex>[r] \ar@<-.4ex>[r] & } Z$, 
and it is \emph{effective} if it is an equalizer of the pair $Y \xymatrix@1 { \ar@<.4ex>[r] \ar@<-.4ex>[r] & } Y \cup_X Y$. 
Every split monomorphism $m$ is regular since if $r$ is a retraction for $m$, then $m$ is an equalizer of $\t{id}$ and $mr$. 
When working in $\sU$, we have the following categorical identification of the closed inclusions.

\begin{lem}  Let $H$ be a subgroup of a topological group $G$, with the subspace topology. The following are equivalent.
\begin{enumerate}[(i)]
\item $H$ is closed in $G$.
\item The inclusion $\io : H \rtarr G$ is an effective monomorphism in $\sU$.
\item The  inclusion $\io: H \rtarr G$ is a regular monomorphism in $\sU$.
\end{enumerate}
Therefore, if $m\colon H\rtarr G$ is a homomorphism and a regular monomorphism in $\sU$, 
then $m$ is isomorphic in $\sU /G$ to the inclusion of a closed subgroup of $G$.
\end{lem}
\begin{proof} For $(i) \Rightarrow (ii)$, since the equivalence relation $E \subset (G \amalg G)^2$ that defines the quotient $q\colon G \amalg G \rtarr G \cup_H G$ is 
closed in $(G \amalg G)^2$, the quotient in $\sU$ is the set theoretic quotient with the quotient topology.  Therefore the equalizer of 
$G \xymatrix@1 { \ar@<.4ex>[r] \ar@<-.4ex>[r] & } G \cup_H G$ is the inclusion $\io: H \rtarr G$. $(ii) \Rightarrow (iii)$ is clear.
$(iii) \Rightarrow (i)$ since equalizers are computed by pulling back a diagonal and diagonals are closed when we work in $\sU $.
For the last statement, let $\io : m(H) \rtarr  G$ be the inclusion of the image of $m$. Then $m$ factors through $\io$ via a  map 
$\til{m} : H \rtarr m(H)$. Since $m$ is regular, it is an equalizer of some pair $f,g : G \xymatrix@1 { \ar@<.4ex>[r] \ar@<-.4ex>[r] & } X$ in $\sU $. Then, since $\til{m}$ is epimorphic and $\io$ is monomorphic, it follows that $\io : m(H) \,\,\, \xymatrix@1 { \ar@{>->}[r] & } G$ is also an equalizer of $f$ and $g$.  Thus $\io: m(H) \,\,\, \xymatrix@1 { \ar@{>->}[r] & } G$ is a regular monomorphism, 
hence is the inclusion of the closed subgroup $m(H)$ of $G$, and $\til{m}$ is a homeomorphism by the uniqueness of equalizers.
\end{proof}

This motivates the following definition.

\begin{defn}\mylabel{Defnclosed}  A closed inclusion $\io\colon H\rtarr G$ of Hopf groups in $\sV$ is a morphism of Hopf groups which is a regular monomorphism in $\sV$.
\end{defn} 

Note that the inclusion $\eta\colon I\rtarr G$ is closed in this sense since it is split by $\epz$.   If we had only required closed inclusions to be effective,
this would not be automatic. Pedantically, we should only consider subgroups up to equivalence of monomorphisms (each factors through the other).  To stay closer to the classical orbit category and to obviate size issues, we agree to choose representatives of equivalence classes when defining the orbit subcategory $\sO_G$ of $\sV$.

\subsection{Families of subgroups}\label{families}

The categorical notion of a sieve suggests a definition of a family of closed subgroups of a Hopf group. A subcategory {$\sS$} of a category $\sC$ is called a sieve if for any object 
$S\in {\sS}$ and any morphism $f\colon C\rtarr S$ in $\sC$, the object $C$ and the morphism $f$ are in the subcategory {$\sS$}.  It follows that {$\sS$} is a full 
subcategory of $\sC$.  

The use of families of subgroups of a topological group $G$ pervades equivariant homotopy theory.  

\begin{lem}  A set of subgroups of  a topological group $G$ is a family if and only if the orbit category $\sO_{\sF}$ is a sieve in $\sO_G$. 
\end{lem}
\begin{proof}Families $\sF$ are closed under subconjugacy, and subconjugacy relations give all morphisms between orbits.
\end{proof}

\begin{defn}  A {\em family} $\sF$ of subgroups of a Hopf group $G$ in $\sV$ is a nonempty set of closed subgroups $H$ such that the orbits $G/H$ are the objects 
of a sieve. 
\end{defn}

\begin{rem} The sieve condition on a family does not play an important role in our present work, and we anticipate that there will be examples in which the sets $\sF$ of interest will not form a family in the preceding sense. However, it is essential that we work with sets $\sF$ that contain $e$, and that we take $\sO_\sF$ to be a full $\sV$-subcategory of $G\ul{\sV}$. Briefly, these two conditions ensure that:
	\begin{enumerate}[(i)]
		\item{}every $\sF$-equivalence is a nonequivariant weak equivalence,
		\item{}the left adjoint in the adjunction $(\bT,\bU)$ may be identified with the functor that restricts a presheaf $P$ to its $G/e$ component and, most importantly,
		\item{}for any $M \in \sM$, the unit $\eta : F_{G/H} \odot M \rtarr \bU\bT(F_{G/H} \odot M)$ is an isomorphism whenever $\bU$ commutes with $(-) \odot M$.
	\end{enumerate}
\end{rem}

\subsection {Double enrichment of equivariant categories}\label{double}
We assume given a cosmos $\sV$, a bicomplete $\sV$-category $\sM$, and a Hopf group $G$ in $\sV$. 
As indicated briefly earlier, we have doubly enriched categories in this context. We explain the relevant categorical framework in more detail here.\footnote{We use the notational conventions of \cite[16.3]{MP} for enriched categories, 
rather than the notations of the categorical literature as in \cite{Kelly}.}  

We have a $\sV$-functor category $G\sM = \mathbf{Fun}(\ul{G},\sM)$ whose objects are the $\sV$-functors $X, Y : \ul{G} \xymatrix@1 { \ar@<.4ex>[r] \ar@<-.4ex>[r] & } \sM$, and whose $\sV$-object of morphisms from $X$ to $Y$ is  the equalizer
	\[
	\ul{G\sM}(X,Y) : = \text{eq} \Big( \ul{\sM}(X,Y) \xymatrix@1 { \ar@<.4ex>[r] \ar@<-.4ex>[r] & } F(G , \ul{\sM}(X,Y)) \Big).
	\]
The two maps being equalized are obtained by transposing the two actions of $G$ on $\ul{\sM}(X,Y)$, and the morphisms of the underlying $\mathbf{Set}$-enriched category $(G\sM)_0$ are the $\sV$-natural transformations. Similar statements hold with $\ul{G}$ replaced by any small $\sV$-category $\sD$.

Using the Hopf structure on $\ul{G}$, we define diagonal actions on products $V \otimes W$ and tensors $V \odot M$ and conjugation actions on homs $\ul{\sM}_G(X,Y)$ and cotensors $F_G(V,Y)$, as in \S\ref{subsec:Hopfactions}. The following lemma indicates that the fixed points of conjugation actions are as one would expect.

\begin{lem}\mylabel{lem:eqlem}For any $V \in G\sV$ and $X,Y \in G\sM$, there are natural isomorphisms
	\[
	\ul{\sM}_G(X,Y)^G \cong \ul{G\sM}(X,Y) \quad\text{and}\quad F_G(V,Y)^G \cong GF(V,Y).
	\]
\end{lem}
\begin{proof} The objects on both sides of these isomorphisms are equalizers, but the pairs of maps being equalized are distinct. However, one may use various adjunctions to show that the same maps equalize both pairs.
\end{proof}

It is also easy to check that the tensor, hom, and cotensor adjunctions between $\sM$ and $\sV$ lift to analogous adjunctions between diagonal and conjugation actions.

\begin{lem}\mylabel{lem:teniso}For any $V \in G\sV$ and $X,Y \in  G\sM$, there are natural isomorphisms
	\[
	\ul{\sM}_G(V \odot X,Y) \cong \ul{\sV}_G(V, \ul{\sM}_G(X,Y)) \cong \ul{\sM}_G(X, F_G(V,Y)).
	\]
\end{lem}
\begin{proof}This follows from the coassociativity of $\psi$ and the fact that $\chi$ is a homomorphism of cocommutative comonoids.
\end{proof}

Applying the functor $(-)^G : G\sV \rtarr \sV$ gives the following result.

\begin{prop}\mylabel{prop:teniso}For any $V \in G\sV$ and $X,Y \in  G\sM$, there are natural $\sV$-isomorphisms
	\[
	\ul{G\sM}(V \odot X,Y) \cong \ul{G\sV}(V, \ul{\sM}_G(X,Y)) \cong \ul{G\sM}(X, F_G(V,Y))
	\]
and set bijections
	\[
	{G\sM}(V \odot X,Y) \cong {G\sV}(V, \ul{\sM}_G(X,Y)) \cong {G\sM}(X, F_G(V,Y)).
	\]
\end{prop}

Now specialize to the case $\sM = \sV$. Then $\odot = \otimes$ and we take $I = G/G$.  We deduce that $G\sV$ is a cosmos.

\begin{thm} The data $( (G\sV)_0 , \otimes , I)$ specifies a cosmos structure on the underlying category of $G\sV$.
\end{thm}
\begin{proof}The symmetric monoidal structure on $\sV$, combined with the cocommutative comonoid structure on $G$, gives rise to a symmetric monoidal structure on $(G\sV)_0$ with the same coherence data as $\sV$. The preceding proposition shows that $\ul{\sV}_G(V,W)$ is the internal hom. The unenriched bicompleteness of $(G\sV)_0$ follows from the $\sV$-bicompleteness of $G\sV$.
\end{proof}

Thus it makes sense to enrich over $G\sV$. In particular, $G\sV$ is enriched over itself using the internal homs $\ul{\sV}_G(V,W)$, and we write 
$\ul{\sV}_G$ for this $G\sV$-enrichment. More generally, we may enrich $(G\sM)_0$ over $G\sV$ using the internal homs $\ul{\sM}_G(X,Y)$.

\begin{thm} \mylabel{Menrich} For any $\sV$-bicomplete $\sV$-category $\sM$, the underlying category $(G\sM)_0$ is enriched over $G\sV$. Its hom objects are the conjugation actions $\ul{\sM}_G(X,Y)$, and we write
$\ul{\sM}_G$ for this enrichment. The category $\ul{\sM}_G$ is $G\sV$-bicomplete.
\end{thm}

\begin{proof}  The construction of a $\sW$-enrichment from a tensoring over $\sW$, such as we have here, is essentially folklore, but see \cite{JK} for a systematic discussion. A more general formulation of this result is given in \cite[Theorem 3.6]{BLV}.
\end{proof}

\subsection {Relationships between the enrichments} 

The category $G\sM = \textbf{Fun}(\ul{G},\sM)$ is enriched in two ways: it has a $\sV$-enrichment $\ul{G\sM}$ by standard enriched category theory, and it has a $G\sV$-enrichment $\ul{\sM}_G$ coming from the Hopf structure on $G$. We recover $G\sM$ from $\ul{G\sM}$ and $\ul{\sM}_G$ by taking underlying sets and $G$-fixed points, respectively.

As expected, we have $\ul{\sM}_G^G \cong \ul{G\sM}$.  That is, the $\sV$-enrichment of $G\sM$ is obtained by taking the $G$-fixed points of the $G\sV$-enrichment: this is precisely \myref{lem:eqlem}. However, we can make a further compatibility statement. There are $\sV$-isomorphisms
	\[
	\ul{G\sM}(X,Y) \cong \ul{\sM}_G(X,Y)^G \cong \ul{G\sV}(I,\ul{\sM}_G(X,Y))
	\]
and hence set bijections
	\[
	G\sV(I,\ul{\sM}_G(X,Y)) \cong G\sM(X,Y) \cong \sV(I, \ul{G\sV}(I,\ul{\sM}_G(X,Y)) ).
	\]
Thus, the construction of $G\sM = (\ul{\sM}_G)_0$ from $\ul{\sM}_G$ factors into two steps: first we apply $(-)^G$ to remove the $G$-action, and then we take $\sV(I,-)$ to remove the $\sV$-structure.
Here we are writing $I$ for both the unit object of $\sV$ and, with trivial action by $G$, the unit object of $G\sV$.

Our experience with $G$-spaces indicates that we may sometimes reverse the order of these operations.   That is, we may recover the set of all $G$-maps $f : X \rtarr Y$ by first ignoring the topology on $\ul{\textbf{Top}}_G$, and then taking the fixed points of the resulting $G$-set. This is not always possible in our present setting, but it is if either
	\begin{enumerate}[(i)]
		\item the enriching category $\sV$ is cartesian closed with with terminal object $*$, and the functor $\sV(*,-) : \sV \to \mathbf{Set}$ is faithful, or
		\item $G$ is discrete, i.e. the image of a group in $\mathbf{Set}$ under the functor $\bI[-] : \mathbf{Set} \to \sV$.
	\end{enumerate}
In the first case, the faithfulness of $\sV(*,-)$ implies that one can check naturality on underlying categories. In the second case, one uses the adjunction $(\bI[-],\sV(*,-))$.

\subsection {The basic adjunction $(\bT,\bU)$}\label{TU}
Suppose $\sF$ is a set of subgroups of $G$ that contains $e$. We describe two different, but equivalent, constructions of the adjunction 
	\[
	\xymatrix@1{\mathbf{Fun} ( \sO^{op}_{\sF},\sM)\ar@<.4ex>[r]^-{\bT} &G \sM \ar@<.4ex>[l]^-{\bU}}
	\]
 considered in \myref{GObjAsPrshvs}. Our second description is quite similar to the construction considered in \cite{GM}, but its equivalence with the first accounts for the usual description of $\bT \colon \mathbf{Fun}(\sO^{op}_\sF,\mathbf{Top}) \xymatrix@1 { \ar@<.4ex>[r] & \ar@<.4ex>[l] } G\mathbf{Top} \colon \bU$ in terms of restriction and fixed points.

Let $\ul{G}$ be a Hopf $\sV$-group $G$ regarded as a $\sV$-category with a single object $*$. We have the Yoneda embedding $\bY : \sC \rtarr \mathbf{Fun}(\sC,\sV)^{op}$ that sends $C \in \sC$ to the   represented functor $\sC(C,-) : \sC \rtarr \sV$. Specializing to the case $\sC = \ul{G}$ yields a $\sV$-functor $\bY : \ul{G} \rtarr G\sV^{op}$ that sends $* \in \ul{G}$ to $\ul{G}(*,-)$. The Yoneda lemma implies that represented functors are free, hence we may identify $\ul{G}(*,-)$ with $G/e$. It follows that $\bY$ factors through the orbit category $\sO^{op}_\sF$, and we obtain a $\sV$-adjunction
	\[
	\bT := \bY^* \colon \xymatrix@1 {\mathbf{Fun} ( \sO^{op}_{\sF},\sM)\ar@<.4ex>[r] & \mathbf{Fun}(\ul{G},\sM) = G \sM \ar@<.4ex>[l]} \colon \text{Ran}_{\bY} = \colon \bU.
	\]
The left adjoint $\bY^*$ restricts a presheaf $P : \sO^{op}_\sF \rtarr \sM$ to the $G$-action on $P(G/e)$, and by abstract nonsense the right adjoint is defined by the equalizer
	\[
	(\text{Ran}_\bY M)(G/H) = \text{eq}\Big( \xymatrix@1 {F(G/H,M) \ar@<.4ex>[r] \ar@<-.4ex>[r] & F(G,F(G/H,M)) } \Big) \cong M^H;
	\]
that is, it sends $M \in G\sM$ to its fixed point presheaf.

Alternatively, we may construct $(\bT,\bU)$ using the techniques in \cite{GM}, taking the functor $\delta : \sO_\sF \to G\sV$ there to be the inclusion. The right adjoints are visibly equal, hence the left adjoints also coincide.

In general,} given ``tensor'' and ``hom'' $\sV$-functors
	\[
		\boxtimes \colon \sP \otimes \sM \rtarr \sN	\quad\text{and}\quad \sM	\ltarr \sP^{op} \otimes \sN \colon \t{hom}
	\]
satisfying a $\sV$-adjunction $\ul{\sN}(P \boxtimes M , N) \cong \ul{\sM}(M , \t{hom}(P,N))$, and a small test diagram $\delta : \sD \rtarr \sP$, we may construct a $(\bT , \bU)$ adjunction by taking
	\[
	\bT := \delta \underset{\sD^{op}}{\boxtimes} (-) \colon \xymatrix@1 { \mathbf{Fun}(\sD^{op},\sM) \ar@<.4ex>[r] & \sN \ar@<.4ex>[l]} \colon \t{hom}(\delta^{op},-) =: \bU.
	\]
In this paper, the $(\bT , \bU)$ adjunction is obtained from
	\[
	\odot \colon G\ul{\sV} \otimes \sM \rtarr G\sM	\quad\text{and}\quad \sM \ltarr G\ul{\sV}^{op} \otimes G\sM \colon GF(\bullet,-)
	\]
and $\delta \colon  \sO_\sF \rtarr G\ul{\sV}$. In \cite{GM}, the $(\bT , \bU)$ adjunction is obtained from
	\[
	\odot \colon \sM \otimes \sV \rtarr \sM	\quad\text{and}\quad \sV \ltarr \sM^{op} \otimes \sM \colon \sM(\bullet,-)
	\]
and a test diagram $\delta : \sD \rtarr \sM$.  In \S \ref{GVModel} we define the adjunction 
$$\bT \colon \mathbf{Fun}(\sF\sD^{op},\sV) \xymatrix@1 { \ar@<.4ex>[r] & \ar@<.4ex>[l] } G\sM \colon \bU$$ 
by taking $\delta$ to be the inclusion $\sF\sD \rtarr G\sM$. Thus the tensor-hom pairings being considered are different, and the test objects are drawn from different categories. Here they are usually objects of $G\sV$, but in \cite{GM}, they are objects of $\sM$.

\end{document}